\documentclass[a4paper, 11pt]{article}  
\usepackage[margin=1in]{geometry}

\usepackage{amsmath, amsthm, amsfonts, amssymb, bbm, mathtools} 
\usepackage{graphicx} 
\usepackage{mathrsfs} 
\usepackage{enumitem} 
\usepackage{titlesec} 
\usepackage{xcolor} 
\usepackage[normalem]{ulem} 

\theoremstyle	{plain}
\newtheorem		{theorem}		{Theorem}		[section]
\theoremstyle	{definition}

\theoremstyle	{definition}
\newtheorem		{defn}			{Definition}	[section]
\newtheorem		{prop}			{Proposition}	[section]
\newtheorem		{corollary}			{Corollary} 	[section]
\theoremstyle	{remark}

\newcommand{ \mc }[1] 			{ \mathcal #1 }
\newcommand{ \mb }[1] 			{ \mathbf #1 }
\newcommand{ \mbb }[1] 			{ \mathbb #1 }
\newcommand{ \norm }[1] 		{ \Vert #1 \Vert }

\newcommand{ \floor }[1] 		{ \lfloor #1 \rfloor }

\newcommand{ \restrictedTo }[1]	{ \big|_{ \!_{ #1 } } }

\newcommand{ \nosquare } 			{ \let\qed\relax }	
\newcommand{ \ldef } 				{ \hspace{ 1 pt } \raisebox{ 0.4 pt }{:} \hspace{ -4 pt }= }

\newcommand{ \matrixEqualsCol } 	{ \hspace{ -3 mm } = \hspace{ -3 mm } } 

\newcommand{ \eps } 		{ \varepsilon }
\newcommand{ \phii } 		{ \zeta }
\newcommand{ \barr } 		{ \overline }
\newcommand{ \indicator } 	{ \mathbbm 1 }
\newcommand{ \prob } 		{ \mathbb P }

\newcommand{ \zeroComp } 		{ \varnothing }
\newcommand{ \compSet } 		{ \mc C }

\newcommand{ \upDownComp } 					{ \mb X } 		
\newcommand{ \leftMostCol }	 				{ Y } 			
\newcommand{ \transKernelComp } 			{ T }			
\newcommand{ \transKernelLeftMostCol } 		{ Q }			
\newcommand{ \distComp } 					{ M } 			
\newcommand{ \distLeftMostCol } 			{ \nu } 		

\newcommand{ \downKernel } 		{ p^\downarrow }
\newcommand{ \upKernel } 		{ p^\uparrow		_{ ( \alpha, \theta ) } }
\newcommand{ \upDownKernel } 	{ \transKernelComp 	^{ ( \alpha, \theta ) } }
\newcommand{ \upDownDist } 		{ \distComp 		^{ ( \alpha, \theta ) } }
\newcommand{ \upDownChain } 	{ \upDownComp 		^{ ( \alpha, \theta ) } }
\newcommand{ \leftMostColChain }{ \leftMostCol 		^{ ( \alpha, \theta ) } }

\newcommand{ \upKernelSym } 			{ p^\uparrow 			_{ ( \alpha, \alpha ) } }
\newcommand{ \upDownKernelSym } 		{ \transKernelComp 		^{ ( \alpha, \alpha ) } }
\newcommand{ \upDownDistSym } 			{ \distComp 			^{ ( \alpha, \alpha ) } }
\newcommand{ \upDownChainSym } 			{ \upDownComp 			^{ ( \alpha, \alpha ) } }
\newcommand{ \leftMostColDistSym } 		{ \distLeftMostCol 		^{ ( \alpha, \alpha ) } }
\newcommand{ \leftMostColChainSym } 	{ \leftMostCol 			^{ ( \alpha, \alpha ) } }

\newcommand{ \upKernelBiased } 				{ p^\uparrow				_{ ( \alpha, 0 ) } }
\newcommand{ \upDownKernelBiased } 			{ \transKernelComp 			^{ ( \alpha, 0 ) } }
\newcommand{ \leftMostColKernelBiased } 	{ \transKernelLeftMostCol 	^{ ( \alpha, 0 ) } }
\newcommand{ \upDownChainBiased } 			{ \upDownComp				^{ ( \alpha, 0 ) } }
\newcommand{ \leftMostColChainBiased } 		{ \leftMostCol 				^{ ( \alpha, 0 ) } }

\newcommand{ \leftMostColLocalProb }{ r }

\newcommand{ \upComp } 						{ \mb C^\uparrow } 		
\newcommand{ \downComp } 					{ \mb C^\downarrow } 	

\newcommand{\setOfOpenSets} 					{ \mc U }
\newcommand{ \bernsteinKernel } 				{ K }
\newcommand{ \eigenIndexSet } 					{ \mbb N \setminus \{ 1 \} }
\newcommand{ \eigenvalueGenerator } 			{ \omega }

\newcommand{ \generalGenerator } 				{ A }
\newcommand{ \generalGeneratorTwo } 			{ B }
\newcommand{ \generalSemigroup } 				{ V_t }
\newcommand{ \generalSemigroupTwo } 			{ W_t }
\newcommand{ \generalDomain } 					{ \mc E }
\newcommand{ \generalDomainTwo } 				{ \mc F }
\newcommand{ \generalSubspaceOfDomain } 		{ E }
\newcommand{ \setOfLambdaValues } 				{ I }

\titleformat*{\section}{\large\bfseries}

\makeatletter
\renewcommand*{\@fnsymbol}[1]{\ifcase#1\or*\else\@arabic{\numexpr#1-1\relax}\fi}
\makeatother

\begin{document}

\title	{
			The leftmost column of ordered Chinese Restaurant Process up-down chains:
			intertwining and convergence
		}
\author{
	Kelvin Rivera-Lopez%
	\thanks{This work was supported in part by NSF grant DMS-1855568.}
	\thanks{Universit\'e de Lorraine, CNRS, IECL, F-54000 Nancy, France, kelvin.rivera-lopez@univ-lorraine.fr} 
		\and 
	Douglas Rizzolo$^*$%
	\thanks{University of Delaware, drizzolo@udel.edu}%
	}
\date{
    \today
}
\maketitle

\abstract{
    Recently there has been significant interest in constructing ordered analogues of Petrov's two-parameter extension of Ethier and Kurtz's infinitely-many-neutral-alleles diffusion model.  One method for constructing these processes goes through taking an appropriate diffusive limit of Markov chains on integer compositions called ordered Chinese Restaurant Process up-down chains.  The resulting processes are diffusions whose state space is the set of open subsets of the open unit interval. In this paper we begin to study nontrivial aspects of the order structure of these diffusions.  In particular, for a certain choice of parameters, we take the diffusive limit of the size of the first component of ordered Chinese Restaurant Process up-down chains and describe the generator of the limiting process.  We then relate this to the size of the leftmost maximal open subset of the open-set valued diffusions.  This is challenging because the function taking an open set to the size of its leftmost maximal open subset is discontinuous.   Our methods are based on establishing intertwining relations between the processes we study.
}

\section{Introduction}
\label{secintro}

The construction and analysis of ordered analogues of Petrov's \cite{Petrov09} two-parameter extension of Ethier and Kurtz's \cite{EthiKurt81} infinitely-many-neutral-alleles diffusion model has recently attracted significant interest in the literature \cite{FPRW20-3, FRSW20-2, krdr2020, RogersWinkel20,ShiWinkel20-2}.  Recall that the ${\tt EKP}(\alpha,\theta)$ diffusions constructed in \cite{Petrov09} are a family of Feller diffusions on the closure of the Kingman simplex
\[\overline{\nabla}_\infty = \left\{\mathbf{x}= (x_1,x_2,\dots) \ : \ x_1\geq x_2\geq \cdots\geq 0, \sum_{i\geq 1} x_i\leq 1\right\}\]
whose generator acts on the unital algebra generated by $\phi_m(\mathbf{x}) = \sum_{i\geq 1}x_i^m$, $m\geq 2$ by 
\[\mathcal{G} = \frac{1}{2}\left(\sum_{i=1}^\infty x_i \frac{\partial^2}{\partial x_i^2}  - \sum_{i,j=1}^\infty x_ix_j \frac{\partial^2}{\partial x_i \partial x_j} - \sum_{i=1}^\infty (\theta x_i+\alpha) \frac{\partial}{\partial x_i}\right).   \] 
In \cite{krdr2020}, for each $ \theta \geq 0 $, $ 0 \leq \alpha < 1 $, and $ \alpha + \theta > 0 $, we constructed a Feller diffusion $\upDownChain$ whose state space $ \setOfOpenSets $ is the set of open subsets of $ ( 0, 1 ) $ such that the ranked sequence of lengths of maximal open intervals in $\upDownChain$ is an ${\tt EKP}(\alpha,\theta)$ diffusion.  This was done by considering the scaling limit of up-down chains associated to the ordered Chinese Restaurant Process.  

While many interesting properties of $\upDownChain$ can be obtained from the corresponding properties for ${\tt EKP}(\alpha,\theta)$ diffusions, properties that depend on the order structure cannot be.  In this paper we begin to study nontrivial aspects of the order structure of these diffusions.  Motivated by \cite[Theorem 2 and Theorem 19]{FPRW20-1} and \cite[Theorem 5]{FPRW18}, which consider similar properties in closely related tree-valued processes, we consider the evolution of the left-most maximal open interval of $\upDownChainBiased$ in running in its $(\alpha,0)$-Poisson-Dirichlet interval partition stationarity distribution.  Recall that the  $(\alpha,0)$-Poisson-Dirichlet interval partition is the distribution of $\{t\in (0,1) : V_{1-t}>0\}$ where $V_t$ is a $(2-2\alpha)$-dimensional Bessel process started from $0$.  We prove the following result.

\begin{theorem}\label{thm main}
Define $\xi \colon  \setOfOpenSets \to [0,1]$ by $\xi(u) = \inf\{s>0 : s\in [0,1]\setminus u\} $.  If $\upDownChainBiased$ is running in its $(\alpha,0)$-Poisson-Dirichlet interval partition stationarity distribution, then $\xi(\upDownChainBiased)$ is a Feller process.  Moreover, the generator of its semigroup $ \mc L \colon \mc D\subseteq C[0,1] \to C[0, 1 ] $ is given by
    $$
    	\mc L f ( x )
    		=
    			x ( 1 - x )
    			f''( x )
    		-	\alpha
    			f'( x )
    $$

    \noindent
    for $x\in (0,1)$, where the domain $ \mc D $ of $\mc L$ consists of functions $ f $ satisfying
    \begin{enumerate}[ label = (D\arabic*) ]   
		\item
        \label{continuous image under generalized diff operator}
        $f \in C^2( 0, 1 ) $ and
        $
			\phii(x)
				=
            			x ( 1 - x )
            			f''( x )
            		-	\alpha
            			f'( x )
        $
        extends continuously to $ [ 0, 1 ] $,
		
        \item
        \label{boundary condition at zero}
        $
        	\int_0^1 
        		( f(x) - f(0) )
        		x^{ - \alpha - 1 } 
        		(1-x)^{ \alpha - 1 }
        		\,
        		dx
        			=
        				0
        			,
        $
        and
        
        \item
		\label{boundary condition at one}
		$ 
        	f'(x) 
        	( 1 - x )^\alpha 
        		\to 
    				0 
    	$ 
    	as $ x \to 1 $.
    	
    \end{enumerate}
\end{theorem}

We consider only the $(\alpha,0)$ case because the known stationary distribution of $\upDownChain$ is an $(\alpha,\theta)$-Poisson-Dirichlet interval partition and, except in the $(\alpha,0)$ case, with probability 1 interval partitions with these distributions do not have left-most maximal open intervals.  We remark that our theorem statement could be slightly simpler if we knew that $\upDownChain$ had a unique stationary distribution, but this is currently an open problem.

Our proof is based on taking the scaling limit of the left-most coordinate in an up-down chain on compositions based on the ordered Chinese Restaurant Process, which are the same chains that were used in \cite{krdr2020} to construct $\upDownChainBiased$.  

\begin{defn}

	For $ n \ge 1 $, a \textit{composition} of $ n $ is a tuple $ \sigma = ( \sigma_1, ..., \sigma_k ) $ of positive integers that sum to $ n $. 
	The composition of $ n = 0 $ is the empty tuple, which we denote by $ \zeroComp $. 
	We write $ | \sigma | = n $ and $ \ell( \sigma ) = k $ when $ \sigma $ is a composition of $ n $ with $ k $ components.
	We denote the set of all compositions of $ n $ by $ \compSet_n $ and their union by 
	$ 
		\compSet 
			= 
				\cup_{ n \ge 0 } \, 
					\compSet_n 
	$.
\end{defn}

An up-down chain on $\compSet_n$ is a Markov chain whose steps can be factored into two parts: 
	1) an up-step from $ \compSet_n $ to $ \compSet_{n+1} $ according to a kernel $p^\uparrow$ followed by 
	2) a down-step from $ \compSet_{n+1} $ to $ \compSet_{n} $ according to a kernel $p^\downarrow$.  
The probability $ T_n(\sigma,\sigma') $ of transitioning from $\sigma$ to $\sigma'$ can then be written as
	\begin{equation}
	\label{eq tupdown} 
		T_n(\sigma,\sigma') 
			= 
				\sum_{\tau\in \compSet_{n+1}} 
					p^\uparrow(\sigma,\tau)
					p^\downarrow(\tau,\sigma').
	\end{equation}
Up-down chains on compositions, and more generally, on graded sets, have been studied in a variety of contexts \cite{BoroOlsh09,FPRW20-1,Fulman09-1,Fulman09-2,RossGan20,Petrov09,Petrov13}, often in connection with their nice algebraic and combinatorial properties.

In the up-down chains we considered, the up-step kernel $\upKernel$ is given by an $(\alpha,\theta)$-ordered Chinese Restaurant Process growth step \cite{PitmWink09}. 
In the Chinese Restaurant Process analogy, we view $\tau=(\tau_1,\dots,\tau_k) \in \compSet_n$ as an ordered list of the number of customers at $k$ occupied tables in a restaurant, so that $\tau_i$ is the number of customers at the $i^{th}$ table on the list.  
During an up-step, a new customer enters the restaurant and chooses a table to sit at according to the following rules:
\begin{itemize}

	\item 
	The new customer joins table $i$ with probability $(\tau_i-\alpha)/(n+\theta)$, resulting in a step from $\tau$ to $(\tau_1,\dots,\tau_{i-1},\tau_i+1,\tau_{i+1},\dots,\tau_k)$.

	\item 
	The new customer starts a new table directly after table $i$ with probability $\alpha/(n+\theta)$, resulting in a step from $\tau$ to $(\tau_1,\dots,\tau_{i-1},\tau_i, 1,\tau_{i+1},\dots,\tau_k)$.
	
	\item 
	The new customer starts a new table at the start of the list with probability $\theta/(n+\theta)$, resulting in a step from $\tau$ to $(1, \tau_1,\tau_2\dots,\tau_k)$.

	\end{itemize}
We note that, for consistency with \cite{FPRW20-3, FRSW20-2}, this up-step is the left-to-right reversal of the growth step in \cite{PitmWink09}.

The down-step kernel $\downKernel$ we consider can also be thought of in terms of the restaurant analogy. 
During a down-step, a seated customer gets up and exits the restaurant according to the following rule: 
\begin{itemize}
\item 
The seated customer is chosen uniformly at random, resulting in a step from $\tau$ to 
\linebreak[4]
$
	( \tau_1, \ldots, \tau_{i-1}, \tau_i - 1, \tau_{ i + 1 }, \ldots, \tau_k )
$
with probability $\tau_i/n$ (the $i^{th}$ coordinate is to be contracted away if $\tau_i-1=0$, or if the $ i^{th} $ table is no longer occupied). 
\end{itemize}
Note that, in contrast to the up-step, the down-step does not depend on $(\alpha,\theta)$.

Let $ (\upDownChain_n(k))_{k\geq 0} $ be a Markov chain on $\compSet_n$ with transition kernel $ \upDownKernel_n$ defined as in Equation \eqref{eq tupdown} using the $\upKernel$ and $\downKernel$ just described. 
A Poissonized version of this chain was considered in \cite{RogersWinkel20,ShiWinkel20-2}.  
It can be shown that $ \upDownChain_n $ is an aperiodic, irreducible chain.  
We denote its unique stationary distribution by $\upDownDist_n$ and note that this is the left-to-right reversal of the $(\alpha,\theta)$-regenerative composition structures introduced in \cite{GnedPitm05}.

The projection $ \phi( \sigma ) = \sigma_1 $ for $ \sigma \neq \zeroComp $ gives rise to the leftmost column processes, defined by 
$ \leftMostColChain_n = \phi( \upDownChain_n ) $.
Let $ \distLeftMostCol_n^{ ( \alpha, \theta ) } = \upDownDist_n \circ \phi^{ -1 } $, the distribution of the leftmost column when the up-down chain is in stationarity.
The following result, interesting in its own right, is a key step in our proof of Theorem \ref{thm main}.

\begin{theorem}
\label{full statement of convergence result}
	For $ n \ge 1 $, let $ \mu_n $ be a distribution on $ \{ 1, \ldots, n \} $.
	Then, for all $ n $, the up-down chain $ \upDownChainBiased_n $ can be initialized so that
	$ \leftMostColChain_n $ is a Markov chain with initial distribution $ \mu_n $.
	Moreover, for any such sequence of initial conditions for $ \upDownChainBiased_n $, if the sequence
	$
		\{
		n^{ -1 }
		\leftMostColChainBiased_n ( 0 )
		\}_{ n \ge 1 }
	$
	has a limiting distribution $ \mu $, then we have the convergence
	$$
		\left(n^{ -1 }
			\leftMostColChainBiased_n (\floor{ n^2 t } )\right)_{t\geq 0}
			\Longrightarrow
				(F( t ))_{t\geq 0}
	$$
	in the Skorokhod space $ D( [ 0, \infty ), [ 0, 1 ] ) $, 
	where $ F $ is a Feller process with generator $ \mc L $ (as in Theorem \ref{thm main}) and initial distribution $ \mu $.
\end{theorem}

While there are many ways to prove a result like Theorem \ref{full statement of convergence result}, we take an approach based on the algebraic properties of the ordered Chinese Restaurant Process up-down chains.  In particular, our proof is based on the following surprising intertwining result.
For a positive integer $ i $ and composition $ \sigma $, we use the notation $ ( i, \sigma ) $ as a shorthand for the composition $ ( i, \sigma_1, \sigma_2, \ldots, \sigma_{ \ell(\sigma)} ) $.

\begin{theorem}
\label{full statement of all intertwining results}
    For $ n \ge 1 $, let $ \Lambda_n $ be the transition kernel from $ \{1,\dots, n \} $ to $ \compSet_n $ given by 
    $$
    	\Lambda_n( i, ( i, \sigma ) ) 	
    			= 
    				\upDownDistSym_{ n - i }( \sigma )
    			,
    $$

	\noindent
	and let $ \bernsteinKernel_n $ be the transition kernel from $ [ 0, 1 ] $ to $ \{1,\dots, n \} $ given by
    $$
    	\bernsteinKernel_n( x, i )
    		=
    			\binom{ n }{ i }
    			x^i
    			( 1 - x )^{ n - i }
    		+	\leftMostColDistSym_n( i )
    			( 1 - x )^n
    		.
    $$
	
	\noindent
    If the initial distribution of
    $ \upDownComp_n^{ ( \alpha, 0 ) } $
    is of the form $ \mu \Lambda_n $ for some distribution $ \mu $ on $ \{1,\dots, n \}  $, then the process 
	$
		\leftMostCol_n^{ ( \alpha, 0 ) } 
	$
	is Markovian.
	In this case, the following intertwining relations hold:
    \begin{enumerate}[ label = (\roman*) ]
    	\item
        $
        	\Lambda_n \upDownKernelBiased_n 	
        		= 	
        				\leftMostColKernelBiased_n \Lambda_n 	
    			,
        $
    	where $ \transKernelLeftMostCol_n^{ ( \alpha, 0 ) } $ is the transition kernel of $ \leftMostCol_n^{ ( \alpha, 0 ) } $, and
	
		\item
        $
        	\bernsteinKernel_n
			e^{ t n ( n + 1 ) ( \transKernelLeftMostCol_n^{ ( \alpha, 0 ) } - \mb 1 ) }
        		= 	
        			U_t
					\bernsteinKernel_n
        $
        for $ t \ge 0 $,
    	where $ U_t $ is the semigroup generated by the operator $ \mc L $ defined in Theorem \ref{thm main}
	and $ \mb 1 $ denotes the identity operator.
	
	\end{enumerate}
	
\end{theorem}

This paper is organized as follows. 
In Section \ref{section leftmost column}, we show that the $ ( \alpha, 0 ) $ leftmost column process is intertwined with its corresponding up-down chain and describe its transition kernel explicitly.
This establishes part of Theorem \ref{full statement of all intertwining results}.
In Section \ref{section general convergence from commutation}, we state a condition under which the convergence of Markov processes can be obtained from some commutation relations involving generators.
In Section \ref{section limit of generator}, we analyze the generator of the limiting process.
In Section \ref{section intertwining leftmost column with limit}, we show that our generators satisfy the commutation relations appearing in the result of Section \ref{section general convergence from commutation}.
In Section \ref{section convergence argument}, we verify the convergence condition appearing in the result in Section \ref{section general convergence from commutation}. 
In Section \ref{section semigroup relation from generator relation}, we provide general conditions under which commutation relations involving generators lead to the corresponding relations for their semigroups. 
Finally in Section \ref{section proofs of main results}, we prove Theorems \ref{thm main}, \ref{full statement of convergence result}, and \ref{full statement of all intertwining results}.


The following will be used throughout this paper.
For a compact topological space $ X $, we denote by $ C( X ) $ the space of continuous functions from $ X $ to $ \mbb R $ equipped with the supremum norm.
Finite topological spaces will always be equipped with the discrete topology.
Any sum or product over an empty index set will be regarded as a zero or one, respectively.
The set of positive integers $ \{ 1, ..., k \} $ will be denoted by $ [ k ] $.
The falling factorial will be denoted using \emph{factorial exponents} -- that is,
$ 
	x^{ \downarrow b } 
		= 
			x ( x - 1 ) 
			\cdot \ldots \cdot 
			( x - b + 1 ) 
$ 
for a real number $ x $ and nonnegative integer $ b $,
and
$ 
	0^{ \downarrow 0 } 
		= 
			1 
$
by convention.
The rising factorial will be denoted by $ (x)_b = x(x+1)\cdots (x+b-1)$.
We denote the gamma function by $ \Gamma ( x ) $.
Multinomial coefficients will be denoted using the shorthand
$$
	\binom	{ | \sigma | }{ \sigma } 	
		= 	
				\begin{cases}
					\displaystyle
					\binom	{ | \sigma | }
							{
							 	\sigma_1, 
							 	..., 
							 	\sigma_{ \ell( \sigma ) } 
							}
						,
						& \sigma \neq \zeroComp,
						\\
					1
						,
						& \sigma = \zeroComp
						.
			\end{cases}
$$

\section{The Leftmost Column Process}
\label{section leftmost column}

Our study of the leftmost column process will be mainly focused on the $ \theta = 0 $ case.
However, it will be useful to study the distribution of the $ ( \alpha, \alpha ) $ leftmost column process when the up-down chain is in stationarity.
As we will see, this distribution has a role in the evolution of the $ ( \alpha, 0 ) $ process.

\begin{prop}

	The stationary distribution of $ \upDownChainSym_n $ is given by
	$$
		\upDownDistSym_n( \sigma ) 	
			=
				\binom{ n }{ \sigma } 
				\frac{ 1 }{ ( \alpha )_n } 
				\prod_{ j = 1 }^{ \ell( \sigma ) } 
					\alpha
					\, 
					( 1 - \alpha )_{ \sigma_j - 1 }
			,
				\qquad 	
				\sigma \in \compSet_n
			,	\,
				n \ge 0
			.
	$$

	\noindent
	Moreover, the following consistency conditions hold:
	\begin{equation}
		\label{eqnConsistency}
		\upDownDistSym_n 	
			= 	
					\upDownDistSym_{ n - 1 }
					\upKernelSym
			= 	
					\upDownDistSym_{ n + 1 }
					\downKernel
			,
					\qquad
					n \ge 1
			.
	\end{equation}

\end{prop}

\begin{proof}
The stationary distribution of $\upDownChain_n$ is identified in \cite[Theorem 1.1]{krdr2020} and the formula in the special case $\alpha=\theta$ follows from \cite[Formula 48]{GnedPitm05}.
The consistency conditions follows from \cite[Proposition 6]{PitmWink09}. 
\end{proof}

\begin{prop}
	\label{stationaryDistLeftMostSym}

	If $ \upDownChainSym_n$ has distribution $ \upDownDistSym_n $, then $ \leftMostColChainSym_n $ has distribution
	$$
		\leftMostColDistSym_n( i ) 	
			=
				\binom{ n }{ i } 
				\frac	{ 
							\alpha 
							\, ( 1 - \alpha )_{ i - 1 } 
						}{ 
							( n - i + \alpha )_i 
						}  
				\,  
				\indicator ( 1 \le i \le n )
			,
				\qquad 	
				i \ge 0
			,	\,
				n \ge 1
			.
	$$

\end{prop}

\begin{proof}

	Let $ 1 \le i \le n $ and $ \sigma \in \compSet_{ n - i } $. 
	It can be verified that
	\begin{equation}
    	\label{eqnCondDistOfRightStructure}
    		\upDownDistSym_n( i, \sigma ) 	
    			= 	
    				\leftMostColDistSym_n( i ) 
    				\upDownDistSym_{ n - i }( \sigma ).
	\end{equation}

	\noindent 
	Summing over $ \sigma $ concludes the proof.
\end{proof}

Let $ n \ge i \ge 1 $ and $ \sigma \in \compSet_{ n - i } $. 
Consider taking an $ ( \alpha, 0) $ up-step from $ ( i, \sigma ) $ followed by a down-step. 
Let $ U $ be the event in which this up-step stacks a box on the first column of $ ( i, \sigma ) $,
 and let $ D $ be the event in which the down-step removes a box from the first column of a composition.
Then, 
	$ \leftMostColLocalProb_{ i, i + 1 } 		= 	\prob ( U \cap D^c ) $, 
	$ \leftMostColLocalProb_{ i, i - 1 } 		= 	\prob ( U^c \cap D ) $, 
	$ \leftMostColLocalProb^{ ( 1 ) }_{ i, i } 	= 	\prob ( U^c \cap D^c ) $, 
	$ \leftMostColLocalProb^{ ( 2 ) }_{ i, i } 	= 	\prob ( U \cap D ) $, and 
	$ \leftMostColLocalProb_{ i, i } 			= 	
													\leftMostColLocalProb^{ ( 1 ) }_{ i, i } 
												+ 	\leftMostColLocalProb^{ ( 2 ) }_{ i, i } $
do not depend on $ \sigma $. Indeed, we have the formulas
\begin{equation}
	\label{yLocalProbDef}
	\begin{matrix}
		\leftMostColLocalProb_{ i, i - 1 } 			
			& 
				\matrixEqualsCol 
			& 
				\frac{ i ( n - i + \alpha ) }{ n ( n + 1 ) }, 					
				\quad
		&
		\leftMostColLocalProb^{ ( 1 ) }_{ i, i } 		
			& 
				\matrixEqualsCol 
			& 	
				\frac{ ( n - i + 1 ) ( n - i + \alpha ) }{ n ( n + 1 ) }, 	
		\\
		\vspace{-2mm}
		\\
		\leftMostColLocalProb_{ i, i + 1 } 			
			& 
				\matrixEqualsCol 
			& 
				\frac{ ( i - \alpha ) ( n - i ) }{ n ( n + 1 ) }, 				
				\quad		
		&
		\leftMostColLocalProb^{ ( 2 ) }_{ i, i } 
			& 
				\matrixEqualsCol 
			& 	
				\frac{ ( i - \alpha ) ( i + 1 ) }{ n ( n + 1 ) }
			.
	\end{matrix} 
\end{equation}

\noindent
We use these formulas to define
$ 
	\leftMostColLocalProb_{ 0, -1 },
$ 
$ 
	\leftMostColLocalProb_{ 0, 1 },
$ 
$ 
	\leftMostColLocalProb^{ ( 1 ) }_{ 0, 0 },
$ 
$ 
	\leftMostColLocalProb^{ ( 2 ) }_{ 0, 0 },
$ 
and
$ 
	\leftMostColLocalProb_{ 0, 0 }
		=
			\leftMostColLocalProb^{ ( 1 ) }_{ 0, 0 }
		+
    	\leftMostColLocalProb^{ ( 2 ) }_{ 0, 0 }.
$
Moreover, we extend
$ 
	\leftMostColLocalProb_{ i, j }
$ 
to be zero for all other integer arguments $ i $ and $ j $.

The following is a useful identity relating the transition kernels of the $ ( \alpha, 0 ) $ and $ ( \alpha, \alpha ) $ chains.

\begin{prop}
	\label{propTransitionRecurrence}

	For $ n \ge 1 $ and $ ( i, \sigma ), ( j, \sigma' ) \in \compSet_n $, we have the identity
	\begin{align*}
		\upDownKernelBiased_n 
			\Big( 
					( i, \sigma ), ( j, \sigma' ) 
			\Big) 	
			& = 	
					\leftMostColLocalProb_{ i, j } 
					\upKernelSym( \sigma, \sigma' ) 
					\indicator( j = i - 1 ) 
				+ 	\leftMostColLocalProb_{ i, j } 
					\downKernel( \sigma, \sigma' ) 
					\indicator( j = i + 1 ) 	
			\\
			& \quad 	
				+	( 	\leftMostColLocalProb^{ ( 1 ) }_{ i, i } 
						\upDownKernelSym_{ n - i }( \sigma, \sigma' ) 
					+ 	\leftMostColLocalProb^{ ( 2 ) }_{ i, i } 
						\indicator( \sigma = \sigma' ) 
					) 
					\indicator( j = i ) 	
			\\
			& \quad 	
				+ 	\leftMostColLocalProb_{ 1, 0 } 
					\upKernelSym( \sigma, ( j, \sigma' ) ) 
					\indicator( i = 1 )
					.
	\end{align*}

\end{prop}

\begin{proof}
	
	Fix $ ( i, \sigma ) $ and $ ( j, \sigma' ) $ in $ \compSet_n $. Let $ \upComp $ be the composition obtained by performing an $ ( \alpha, 0 ) $ up-step from $ ( i, \sigma ) $ and $ \downComp $ be the composition obtained by performing a down-step from $ \upComp $. 
	As before, let $ U $ be the event in which the up-step adds to the first column of a composition 
	and 
	$ D $ be the event in which the down-step removes from the first column of a composition. 
Then, we have that
	$$
		U 
			= 
				\big
				\{ 
					\upComp = ( i + 1, \sigma ) 
				\big
				\}
			,
			\quad
		U^c
			= 
				\{ 
					\upComp_1 = i 
				\}
			,
			\quad
		D^c
			\subset 
				\{ 
					\downComp_1 = \upComp_1 
				\}
			,
	$$
	and
	\begin{align*}
		D
			& \subset 
				\Big
				\{ 
					\upComp_1 > 1
				,
					\,
					\downComp = ( \upComp_1 - 1, ( \upComp )_2^{ \ell( \upComp ) } ) 
				\Big
				\}
			\cup
				\Big
				\{ 
					\upComp_1 = 1
				,
					\,
					\downComp = ( \upComp )_2^{ \ell( \upComp ) }
				\Big
				\}
			.
	\end{align*}

	To obtain the identity, we note that
	$$
		\upDownKernelBiased_n( ( i, \sigma ), ( j, \sigma' ) )
			= 	\prob \{  \downComp = ( j, \sigma' )  \},
	$$
	
	\noindent 
	and rewrite this probability by conditioning on the above sets.
	Of particular importance will be the following observations:
	the conditional distribution of 
	$ 
		( \downComp )_2^{ \ell( \downComp ) } 
	$ 
	given 
	$ 
		( \upComp )_2^{ \ell( \upComp ) } 
	$ 
	and $ D^c $ is 
	$ 
		\downKernel( ( \upComp )_2^{ \ell( \upComp ) }, \cdot \, )
		,
	$ 
	and conditionally given $ U^c $, 
	$ 
		( \upComp )_2^{ \ell( \upComp ) } 
	$ 
	is independent of $ D $ and has distribution 
	$ 
		\upKernelSym( \sigma, \cdot \, ) 
		.
	$
	We also make use of the fact that the events
	$
		\{ 
			\upComp = ( n + 1 - | \rho |, \rho )
		\}
	$
	and
	$
		\{ 
			( \upComp )_2^{ \ell( \upComp ) }
				= 
					\rho
		\}
	$
	are identical, since the size of $ \upComp $ is known to be $ n + 1 $.
		
	Our first conditional probability is given by
	\begin{align*}
		\prob	( 
					\downComp = ( j, \sigma' ) 
				| 	U, 
					D 
				)
			& = 	\prob	( 
								\downComp = ( j, \sigma' ) 
							| 	\upComp = ( i + 1, \sigma ), 
								D 
							) 	
			\\
			& = 	\prob	( 
								( i, \sigma ) = ( j, \sigma' ) 
							| 	\upComp = ( i + 1, \sigma ), 
								D 
							) 	
			\\
			& = 	\indicator ( ( j, \sigma' ) = ( i, \sigma ) ).
	\end{align*}

	\noindent 
	Next, we will condition on $ U \cap D^c $.
	Notice that this is a null set when $ i = n $.
	When $ i < n $, we have
	\begin{align*}
		\prob	( 
					\downComp = ( j, \sigma' ) 
				| 	U, 
					D^c 
				) 		
			& = 	
					\prob	( 
								\upComp_1 = j, 
								( \downComp )_2^{ \ell( \downComp ) } = \sigma' 
							| 	\upComp = ( i + 1, \sigma ), 
								D^c 
							) 		
			\\
			& = 	
					\indicator ( j = i + 1 ) 
					\prob	( 
								( \downComp )_2^{ \ell( \downComp ) } = \sigma' 
							| 	( \upComp )_2^{ \ell( \upComp ) } = \sigma, 
								D^c 
							) 			
			\\
			& = 	\indicator ( j = i + 1 ) 
					\downKernel( \sigma, \sigma' ).
	\end{align*}

	\noindent 
	Conditioning on $ U^c \cap D $ will require two cases.
	For $ i > 1 $, we have
	\begin{align*}
		\prob	( 
					\downComp = ( j, \sigma' ) 
				| 	U^c, 
					D 
				) 		
			& = 	
					\prob	( 
								\downComp = ( j, \sigma' ) 
							| 	\upComp_1 = i, 
								D 
							) 		
			\\
			& = 	
					\prob	( 
								( 
									i - 1, 
									( \upComp )_2^{ \ell( \upComp ) }  
								) 
							= 	( j, \sigma' ) 
							| 	\upComp_1 = i, 
								D 
							)
			\\
			& = 	
					\indicator ( j = i - 1 ) 
					\prob	( 
								( \upComp )_2^{ \ell( \upComp ) } 
							= 	\sigma' 
							| 	U^c, 
								D 
							) 	
			\\
			& = 	
					\indicator ( j = i - 1 ) 
					\prob	( 
								( \upComp )_2^{ \ell( \upComp ) } 
							= 	\sigma' 
							| 	U^c 
							)
			\\
			& = 	
					\indicator ( j = i - 1 ) 
					\upKernelSym( \sigma, \sigma' ),
	\end{align*}

	\noindent 
	and for $ i = 1 $, we have
	\begin{align*}
		\prob	( 
					\downComp = ( j, \sigma' ) 
				| 	U^c, 
					D 
				) 		
			& = 	
					\prob	( 
								\downComp = ( j, \sigma' ) 
							| 	\upComp_1 = 1, 
								D  
							) 	
			\\
			& = 	
					\prob	( 
								( \upComp )_2^{ \ell( \upComp ) } 
							= 	( j, \sigma' )
							| 	U^c, 
								D 
							) 	
			\\
			& = 	
					\prob	( 
								( \upComp )_2^{ \ell( \upComp ) } 
							= 	( j, \sigma' ) 
							| 	U^c 
							) 		
			\\
			& = 	
					\upKernelSym( \sigma, ( j, \sigma' ) ).
	\end{align*}

	\noindent 
	Finally, we condition on $ U^c \cap D^c $.
	We have that
	\begin{align*}
		& \prob	\big( 
					\downComp = ( j, \sigma' ) 
				| 	U^c, 
					D^c 
				\big) 	
			\\
			& = 	
					\prob	\Big( 
								\upComp_1 = j, 
								( \downComp )_2^{ \ell( \downComp ) } = \sigma' 
							| 	\upComp_1 = i, 
								D^c 
							\Big) 		
			\\
			& = 	
					\indicator ( j = i ) 
					\,
					\prob	\Big( 
								( \downComp )_2^{ \ell( \downComp ) } = \sigma' 
							| 	U^c, 
								D^c 
							\Big) 	
			\\
			& = 	
					\indicator ( j = i ) 
					\!\!\!\!
					\sum_{ \tau \in \compSet_{ n + 1 - i }  }  
						\!\!\!\!\!
						\prob	\Big( 
									( \upComp )_2^{ \ell( \upComp ) } = \tau 
								| 	U^c, 
									D^c 
								\Big)  
					\, \prob 	\Big( 
									( \downComp )_2^{ \ell( \downComp ) } = \sigma' 
								| 	\upComp = ( i, \tau ), 
									D^c
								\Big) 		
			\\
			& = 	
					\indicator ( j = i ) 
					\!\!\!\!
					\sum_{ \tau \in \compSet_{ n + 1 - i }  }  
						\!\!\!
						\!\!
						\prob	\Big( 
									( \upComp )_2^{ \ell( \upComp ) } = \tau 
								| 	U^c
								\Big)  
					\,	\prob 	\Big( 
									( \downComp )_2^{ \ell( \downComp ) } = \sigma' 
								| 	( \upComp )_2^{ \ell( \upComp ) } = \tau, 
									D^c 
								\Big)
			 \\
			& = 	
					\indicator ( j = i ) 
					\!\!\!
					\sum_{ \tau \in \compSet_{ n + 1 - i }  }  
						\!\!\!
						\upKernelSym( \sigma, \tau )  
						\downKernel( \tau, \sigma' ) 		
			\\
			& = 	
					\indicator ( j = i ) 
					\,
					\upDownKernelSym_{ n - i }( \sigma, \sigma' ).
	\end{align*}

	\noindent Collecting the terms above with the appropriate terms in (\ref{yLocalProbDef}) establishes the result. 
\end{proof}

%
%
%
%

Let $ n \ge 1 $. 
We define a transition kernel $ \Lambda_n $ from $ [ n ] $ to $ \compSet_n $ by 
$$
	\Lambda_n( i, ( i, \sigma ) ) 	
			= 
					\upDownDistSym_{ n - i }( \sigma )
					, 
$$

\noindent
and a transition kernel $ \Phi_n $ from $ \compSet_n $ to $ [ n ] $ by 
$$
	\Phi( \sigma, i )
			= 
					\indicator( \sigma_1 = i )
					.
$$

\begin{prop}
	\label{propIntertwining}

	For $ n \ge 1 $, the transition kernel 
    $ 
    	\leftMostColKernelBiased_n 
    		= 
    			\Lambda_n 
    			\upDownKernelBiased_n 
    			\Phi_n 
    $
	satisfies
    \begin{equation}
    	\label{eqnIntertwining}
        	\Lambda_n \upDownKernelBiased_n 	
        		= 	
        				\leftMostColKernelBiased_n 
						\Lambda_n
				.
    \end{equation}

	\noindent
    Consequently, if the initial distribution of $ \upDownComp_n^{ ( \alpha, 0 ) } $ is of the form $ \mu \Lambda_n $,
    then $ \leftMostColChainBiased_n $ is a time-homogeneous Markov chain with transition kernel $ \leftMostColKernelBiased_n $.
	Moreover, the transition kernel $ \leftMostColKernelBiased_n $ is given explicitly by
	$$
		\leftMostColKernelBiased_n( i, j ) 	
			= 	
				\leftMostColLocalProb_{ i, j } 
			+ 	\leftMostColLocalProb_{ 1, 0 } 
				\leftMostColDistSym_n( j ) 
				\indicator( i = 1 )
			.
	$$

\end{prop}

\begin{proof}

	Let $ C_n $ be the kernel on $ [ n ] $ defined by the right side of the above equation.	
	Fix $ i, j \in [ n ] $ and $ \sigma' \in \compSet_{n-j} $. 
	Using Proposition \ref{propTransitionRecurrence} and the identities (\ref{eqnConsistency}) and (\ref{eqnCondDistOfRightStructure}), we compute
	\begin{align*}
		&
		( \Lambda_n \upDownKernelBiased_n ) 
		( i, ( j, \sigma' ) ) 		
		\\
			& \hspace{ 10 mm } = 	
					\sum_{ \sigma \in \compSet_{ n - i } } 
						\Lambda_n( i, ( i, \sigma ) ) 
						\upDownKernelBiased_n( ( i, \sigma ), ( j, \sigma' ) ) 	
			\\ 	
			& \hspace{ 10 mm } = 	
					\leftMostColLocalProb_{ i, j } 
					\!\!
					\sum_{ \sigma \in \compSet_{ n - i } } 
						\!\!
						\upDownDistSym_{ n - i }( \sigma ) 
						\left(
        						\upKernelSym( \sigma, \sigma' )
								\indicator( j = i - 1 )
							+ 	\downKernel( \sigma, \sigma' )
								\indicator( j = i + 1 )
						\right)
			\\
			& \hspace{ 10 mm } \quad 	
				+ 	\indicator( j = i )
					\!\!
					\sum_{ \sigma \in \compSet_{ n - i } } 
						\!\!
						\upDownDistSym_{ n - i }( \sigma ) 
    					\left(
        						\upDownKernelSym_{ n - i }( \sigma, \sigma' ) 	
        						\leftMostColLocalProb^{ ( 1 ) }_{ i, i } 
        					+ 	\indicator( \sigma = \sigma' ) 		
        						\leftMostColLocalProb^{ ( 2 ) }_{ i, i } 
    					\right)
			\\
			& \hspace{ 10 mm } \quad 	
				+ 	\leftMostColLocalProb_{ 1, 0 } 
					\indicator( i = 1 )
					\!\!
					\sum_{ \sigma \in \compSet_{ n - i } } 
						\!\!
						\upDownDistSym_{ n - i }( \sigma ) 
						\upKernelSym( \sigma, ( j, \sigma' ) ) 	
			\\
			& \hspace{ 10 mm } = 	
					\leftMostColLocalProb_{ i, j } 
					\left(
       						\upDownDistSym_{ n - j }( \sigma' ) 
							\indicator( j = i - 1 )
						+ 	\upDownDistSym_{ n - j }( \sigma' ) 
							\indicator( j = i + 1 )
					\right)
			\\
			& \hspace{ 10 mm } \quad 	
				+ 	\indicator( j = i )
					\left(
    						\upDownDistSym_{ n - j }( \sigma' ) 
    						\leftMostColLocalProb^{ ( 1 ) }_{ i, i } 
    					+ 	\upDownDistSym_{ n - j }( \sigma' ) 
    						\leftMostColLocalProb^{ ( 2 ) }_{ i, i } 
					\right)
				+ 	\leftMostColLocalProb_{ 1, 0 } 
					\indicator( i = 1 ) 	
					\upDownDistSym_n( j, \sigma' ) 		
			\\
			& \hspace{ 10 mm } = 	
                    \leftMostColLocalProb_{ i, j } 
					\upDownDistSym_{ n - j }( \sigma' ) 
				+ 	\leftMostColLocalProb_{ 1, 0 } 
					\indicator( i = 1 ) 
					\leftMostColDistSym_n( j ) 
					\upDownDistSym_{ n - j }( \sigma') 
			\\
			& \hspace{ 10 mm } = 	
					C_n( i, j ) 
					\Lambda_n( j, ( j, \sigma' ) ) 	
			\\
			& \hspace{ 10 mm } = 	
					( C_n \Lambda_n )( i, ( j, \sigma' ) ).
	\end{align*}

	\noindent 
	The final equality follows from the fact that $ \Lambda_n( j, \, \cdot \, ) $ is supported on 
	$ 
		\{ 
			\sigma \in \compSet_n
		:	\sigma_1 = j
		\}
	$. 
	This establishes the identity 
	$ 
		\Lambda_n 
		\upDownKernelBiased_n 	
			= 	
				C_n 
				\Lambda_n 
			.
	$
	Observing that $ \Lambda_n \Phi_n $ is the identity kernel on $ [ n ] $, we find that
	
%
	\begin{equation*}
	 	\leftMostColKernelBiased_n
	 		= 	
	 			\Lambda_n 
	 			\upDownKernelBiased_n 
	 			\Phi_n 	
	 		= 	
				C_n 
	 			\Lambda_n 
	 			\Phi_n 
	 		= 	
				C_n
			,
	\end{equation*}

	\noindent
	from which we obtain (\ref{eqnIntertwining}) and the explicit description of $ \leftMostColKernelBiased_n $.
	The final claim follows from applying Theorem 2 in \cite{rogersPitman1981}.
\end{proof}

\section{Convergence from Commutation Relations}
\label{section general convergence from commutation}

In this section, we provide a condition under which commutation relations between operators implies the convergence of those operators in an appropriate sense.
In the interest of generality, we first state this condition in the setting of Banach spaces, but we then reformulate it in the context of Markov processes to suit our purposes.
The general setting is as follows.

Let $ V, V_1, V_2, \ldots $ be Banach spaces and $ \pi_1, \pi_2, \ldots $ be uniformly bounded linear operators with $ \pi_n \colon V \to V_n $.
These spaces will be equipped with the following mode of convergence. 
\begin{defn} 
\label{convergenceInDifferentSpaces}
    A sequence $ \{ f_n \}_{ n \ge 1 } $ with $ f_n \in V_n $ converges to an element $ f \in V $ (and we write $ f_n \to f $) if
    $$
    	\norm{
    		f_n
    	-	\pi_n
    		f
    	}
    		\xrightarrow[ n \to \infty ]{}
    			0
    		,
    $$

    \noindent
    where for convenience, we denote every norm by the same symbol $ \norm{ \cdot } $.

\end{defn}

\begin{prop}
	\label{convergenceFromIntertwining}

	For $ n \ge 1 $, let 
	$ L_n \colon D_n \subset V \to V_n $ 
	and $ A_n \colon V_n \to V_n $
	be linear operators 
	in addition to $ A \colon D \subset V \to D $.
	Suppose that for every $ f \in D $, 
	\begin{enumerate}[ label = (\roman*) ]
		\item
		\label{generalIntertwiningCondition}
		$
			A_n
			L_n
			f
				=
					L_n
					A
					f
		$
		for large $ n $, and

		\item
		$
			(
				L_n
			-	\pi_n
			)
				f
			\longrightarrow
				0
		$
		as $ n \to \infty $ (the sequence need only be defined for large $ n $).
		
	\end{enumerate}

	\noindent
	Then for $ f \in D $, the sequence $ f_n = L_n f $ (defined for large $ n $) satisfies
    $$
    	f_n
    		\longrightarrow
    			f
    	\qquad
    	\text{and}
    	\qquad
       	A_n
		f_n
          		\longrightarrow
           			A
        			f
				.
    $$

\end{prop}

\begin{proof}

	Let $ f \in D $ and $ n $ be large enough so that \ref{generalIntertwiningCondition} holds.
	In particular, we can define $ f_n = L_n f $. 
	Writing
	$$
		\norm{ f_n - \pi_n f }
			= 
				\norm{
        				L_n
        				f
					-	\pi_n
						f
					}
			,
	$$
	
	\noindent
	it is clear that $ f_n \to f $. 
	Writing
	\begin{align*}
		\norm{ 
    		A_n
			f_n
		- 	\pi_n 
			A
			f 
		}
			& = 
            		\norm{ 
        				A_n
						L_n
						f
            		- 	\pi_n
						A
            			f 
            		}
			\\
			& = 
            		\norm{ 
        				L_n
						A
						f
            		- 	\pi_n
						A
            			f 
            		}
			\\
			& = 
            		\norm{ 
            			(
    						L_n
                		- 	\pi_n
                		)
						A
            			f 
            		}		
	\end{align*}

	\noindent
	and noting that $ A f \in D $, we obtain the other convergence.
\end{proof}

In the probabilistic context, the above result has some additional consequences.
	
\begin{theorem}
	\label{markovConvergenceFromIntertwining}
	
	Let
    	$ E $ be a compact, separable metric space, 
    	$ A $ be the generator of the Feller semigroup $ S( t ) $ on $ C( E ) $, and
    	$ D $ be a core for $ A $ that is invariant under $ A $.
	For each $ n \ge 1 $, let
    	$ E_n $ be a finite set endowed with the discrete topology, 
    	$ Z_n $ be a Markov chain on $ E_n $,
    	$ \gamma_n \colon E_n \to E $ be any function, and
    	$ L_n \colon D_n \subset C( E ) \to C( E_n ) $ be a linear operator.
    Denote
    	the transition operator of $ Z_n $ by $ S_n $
    	and
    	the projection $ f \mapsto f \circ \gamma_n $ by $ \pi_n \colon C( E ) \to C( E_n ) $.
Let
	$ \{ \delta_n \}_{ n \ge 1 } $ and $ \{ \eps_n \}_{ n \ge 1 } $ be positive sequences converging to zero such that
	$ 
		\eps_n^{ -1 } 
		\delta_n
			\to
				1
			.
	$
	Suppose that for $ f \in D $, the following statements hold:
	\begin{enumerate}[ label = (\alph*) ]
		\item
		\label{generator relation in general result}
		$
			\delta_n^{ - 1 }
			( S_n - \mb 1 )
			L_n
			f
				=
					L_n
					A
					f
		$
		for large $ n $, and

		\item
		\label{convergence condition in general result}
		$
			(
				L_n
			-	\pi_n
			)
				f
			\longrightarrow
				0
		$
		as $ n \to \infty $ (the sequence need only be defined for large $ n $).
		
	\end{enumerate}

	\noindent
	Then, 
    \begin{enumerate}[ label = (\roman*) ]
    	
    	\item
    	\label{semigroupConvergenceInGeneralTheorem}
    	the discrete semigroups 
		$ 
			\{ 
				1, S_n, S_n^2, ... 
			\}_{ n \ge 1 } 
		$ 
		converge to 
		$ \{ S(t) \}_{ t \ge 0 } $ in the following sense: 
		for all $ f \in C( E ) $ and $ t \ge 0 $, 
    	$$
			S_n^{ \floor{ t/\eps_n } }
			\pi_n
			f
				\xrightarrow[ n \to \infty ]{}
        			S(t)
    				f
    	$$

    	\item
    	\label{uniformConvergenceInGeneralTheorem}
    	the above convergence is uniform in $ t $ on bounded intervals, and

    	\item
    	\label{pathConvergenceInGeneralTheorem}
    	if $ A $ is conservative and the distributions of
    	$
    		\gamma_n( Z_n( 0 ) )
    	$
    	converge, say to $ \mu $,
    	then we have the convergence of paths
    	$$
    		\gamma_n( Z_n \floor{ t/\eps_n } )
    			\Longrightarrow
    				F( t )
    	$$
    	in the Skorokhod space $ D( [ 0, \infty ), E ) $, where $ F( t ) $ is a Feller process with initial distribution $ \mu $ and generator $ A $.

    \end{enumerate}

\end{theorem}

\begin{proof}

	This is a combination of Proposition \ref{convergenceFromIntertwining} and standard convergence results.  In particular, for $ f \in D $, we can define the sequence $ f_n = L_n f $ for large $ n $ and obtain the convergence
    $$
    	f_n
    		\longrightarrow
    			f
    	\quad
		\qquad
    		\text{and}
    	\quad
		\qquad
			\delta_n^{ - 1 }
			( S_n - \mb 1 )
			f_n
        		\longrightarrow
           			A
        			f
			.
    $$
    
    \noindent
	Recalling that
	$ 
		\eps_n^{ -1 } 
		\delta_n
			\to
				1
			,
	$
	we then obtain the convergence
    $
		\eps_n^{ - 1 }
		( S_n - \mb 1 )
		f_n
      		\to
				A
       			f
			.
    $
    Applying Chapter 1 Theorem 6.5 in \cite{EthKurtzBook} then yields the convergence of semigroups in \ref{semigroupConvergenceInGeneralTheorem} and \ref{uniformConvergenceInGeneralTheorem}. 
	Applying Chapter 4 Theorem 2.12 in \cite{EthKurtzBook} yields the path convergence in \ref{pathConvergenceInGeneralTheorem}.
\end{proof}

\section{The Limiting Generator}
\label{section limit of generator}

In this section, we introduce the generator of a Feller process on $ [ 0, 1 ] $ that will be identified as the limiting process.
We describe this generator both on a core of polynomials and on its full domain.
However, the core description is sufficient for the analysis that will follow.

Let $ \mc P $ denote the space of polynomials on $ [ 0, 1 ] $ equipped with the supremum norm.
We will study the operator $ \mc B \colon \mc P \to \mc P $ and the functional $ \eta \colon \mc P \to \mbb R $ given by
$$
	(
    	\mc B 
    	f
	)
	( x )
		=
			x (1-x)
			f''(x)
		-
			\alpha
			f'(x)
		,
		\quad
		x \in [ 0, 1 ]
		,
$$

\noindent
and
\begin{align}
	\label{functionalUsingDerivative}
	\eta(f)
		& \ldef
           	\int_0^1 
           		( f(x) - f(0) )
           		x^{ - \alpha - 1 } 
           		(1-x)^{ \alpha - 1 }
           		\,
           		dx
		\nonumber
		\\
		& = 
        	\int_0^1 
        		f'(x)
        		x^{ - \alpha } 
        		(1-x)^{ \alpha }
        		\alpha^{-1}
					\,
        		dx
		.
\end{align}

\noindent
Letting
$
	\mbb N
		= 
			\{ 0, 1, 2, \ldots \}
		,
$
we define a family of polynomials 
$ 
	\{
		h_n
	\}_{ n \in \eigenIndexSet }
$
by
$$
	h_n( x )
		= 
    			\sum_{ s = 0 }^n
    				\,
    				x^s
       				( -1 )^{ n - s }
					\frac 	{
								( n - 1 )_s
							}{
								s!
							}
					\frac 	{
            					( s - \alpha )_{ n - s }
							}{
								( n - s )!
							}
		,
			\qquad
			x \in [ 0, 1 ]
		.
$$

\noindent
Note that $ h_0 \equiv 1 $ and $ h_n $ has degree $ n $.
Moreover, these polynomials are related to 
	the Jacobi polynomials $ P_n^{ ( a, b ) } $
	and 
	the shifted Jacobi polynomials $ J_n^{ ( a, b ) } $ 
\cite{rababah2004, szego75} by the identity
$$
	h_n( x )
		=
 				J_n^{ ( \alpha - 1, - \alpha - 1 ) }( x )
		= 
 				P_n^{ ( \alpha - 1, - \alpha - 1 ) }( 2 x - 1 )
		,
			\qquad
			x \in [ 0, 1 ]
		.
$$

\begin{prop}
    \label{propertiesOfB}

    Let
    $
		\mc H
    		=
				\ker
				\eta
	$
    and
    $
    	\eigenvalueGenerator_n
    		=
    			- n ( n - 1 )
    $
    for 
	$
		n \in \eigenIndexSet
		.
	$
	The following statements hold:
    \begin{enumerate}[ label = (\roman*) ]
    
    	\item
		\label{eigenfunctionsOfB}
		$ \mc B h_n = \eigenvalueGenerator_n h_n $ for all $ n \in \eigenIndexSet $,
		
		\item
		\label{eigenbasisOfH}
		the family 
		$ 
			\{
				h_n
			\}_{ n \in \eigenIndexSet }
		$
		is a Hamel basis for $ \mc H $, and
		
		\item
		\label{densityOfH}
		$ \mc H $ is a dense subspace of $ C[ 0, 1 ] $.
    \end{enumerate}

\end{prop}

\begin{proof}

    The claim in \ref{eigenfunctionsOfB} can be obtained from the classical theory of Jacobi polynomials 
    (e.g. (4.1.3), (4.21.2), and (4.21.4) in \cite{szego75}).
    
    Noting that $ h_n $ has degree $ n $ shows that the family
    	$ 
    		\{ h_n \}_{ n \in \eigenIndexSet } 
    	$ 
    is independent.
    Since $ h_0 \equiv 1 $, it clearly lies in $ \mc H $.
    To see that the other $ h_n $ also lie in $ \mc H $, we use \ref{eigenfunctionsOfB} to identify them as elements in the range of $ \mc B $ and observe that this range lies in $ \mc H $.
    Indeed, this can be verified using (\ref{functionalUsingDerivative}): for $ f \in \mc P $, we have that
    \begin{align*}
    	\eta( 
        		\mc B
        		f
    		)
    			& =
                       	\int_0^1 
                       		(
                    			x (1-x)
                    			f''(x)
                    		-
                    			\alpha
                    			f'(x)
    						+
    							\alpha
    							f'(0)							
    						)
                       		x^{ - \alpha - 1 } 
                       		(1-x)^{ \alpha - 1 }
                       		\,
                       		dx
        		\\
    			& =
                       	\int_0^1 
                    		f''(x)
    						x^{ - \alpha } 
                       		(1-x)^{ \alpha }
                       		\,
                       		dx                   		
                    -	\alpha
                       	\int_0^1 
                    		( f'(x) - f'(0) )
    						x^{ - \alpha - 1 } 
                       		(1-x)^{ \alpha - 1 }
                       		\,
                       		dx
        		\\
    			& =
    					\alpha
    					\,
    					\eta( f' )
    				-	\alpha
        				\,
    					\eta( f' )
    			\\
    			& =
    					0
    			.
    \end{align*}

    \noindent
    To obtain equality from the containment
    $ 
    	\text{span}
    	\{
    		h_n
    	\}_{ n \in \eigenIndexSet }
    		\subset
    			\mc H
    		,
    $
    we observe that the former space is a maximal subspace of $ \mc P $ (it has codimension one) while the latter is a proper subspace of $ \mc P $.

    The claim in \ref{densityOfH} will follow from showing that $ \eta $ is not continuous (see Chapter 3 Theorem 2 in \cite{bollobas}).
    To see that this holds, notice that the functions
    $
    	f_j(x)
    		=
    			( 1 - x )^j
    		,
    			\,
    			j \ge 1
    		,
    $
    have norm 1 but their images under $ \eta $ are unbounded:
    \begin{align*}
    	\eta( f_j )
    		& =
            	-
    			\int_0^1 
            		j
    				x^{ - \alpha } 
            		( 1 - x )^{ j - 1 + \alpha }
            		\alpha^{ -1 }
    					\,
            		dx
    		\\
    		& =
            	-
    			\frac	{ 
    						\Gamma( 1 - \alpha ) 
    						\Gamma( j + \alpha ) 
    					}{ 
    						\alpha 
    						\,
    						\Gamma( j )
    					}
    		.
    		\qedhere
    \end{align*}
\end{proof}

\begin{prop}

	The operator $ \mc B\vert_{ \mc H } $ is closable and its closure, $ \barr{ \mc B\vert_{ \mc H } } $, is the generator of a Feller semigroup on $ C[ 0, 1 ] $.
	
	\label{existence of the generator}
\end{prop}

\begin{proof}

    We show that $ \mc B\vert_{ \mc H } $ satisfies the conditions of the Hille-Yosida Theorem.
    For $ \lambda > 0 $, Proposition \ref{propertiesOfB}\ref{eigenfunctionsOfB}-\ref{eigenbasisOfH} show that the range of 
    $ 
    	\lambda - \mc B\vert_{ \mc H }
    $
    is exactly $ \mc H $.
    Proposition \ref{propertiesOfB}\ref{densityOfH} then tells us that 
    this range, as well as the domain of
    $ 
    	\mc B\vert_{ \mc H } 
    $,
    is dense in $ C[ 0, 1 ] $.

    To establish the positive-maximum principle, suppose that
    $ 
    	f 
    		\in 
    			\mc H 
    $
    has a nonnegative maximum at $ y \in [ 0, 1 ] $.
    If $ y \neq 0 $, the tools of differential calculus show that
    $
    	( 
    	\mc B\restrictedTo{ \mc H } 
    	f
    	)( y )
    		\le 
    			0
    		,
    $
    as desired.
    When $ y = 0 $, consider the element $ F \in L^1[ 0, 1 ] $ given by
    $$
    	F( x )
    		=
    			( f( x ) - f( 0 ) )
           		x^{ - \alpha - 1 } 
           		(1-x)^{ \alpha - 1 }			
    $$

    \noindent
    almost everywhere. 
    Since $ f(x) \le f( 0 ) $ on $ [ 0, 1 ] $, the norm of $ F $ is given by
    \begin{align*}
    	\norm F
    	_1
    		& = 
               	\int_0^1 
               		| f( x ) - f( 0 ) |
               		x^{ - \alpha - 1 } 
               		(1-x)^{ \alpha - 1 }
               		\,
               		dx
    		\\
    		& = 
               	-
    			\int_0^1 
               		( f( x ) - f( 0 ) )
               		x^{ - \alpha - 1 } 
               		(1-x)^{ \alpha - 1 }
               		\,
               		dx
    		\\
    		& = 
    			-
    			\eta( f )
    		.			
    \end{align*}

    \noindent
    Recalling that $ f \in \mc H = \ker \eta $, it follows that 
    $
    	F = 0
    $
    almost everywhere.
    Together with the continuity of $ f $, this implies that 
    $ 
    	f 
    		\equiv 
    			f( 0 ) 
    		,
    $
    and consequently,
    $
    	( 
    	\mc B\restrictedTo{ \mc H } 
    	f
    	)( y )
    		\le
    			0
    		.
    $
\end{proof}

The final result in this section is the explicit description of the generator $ \barr{ \mc B\vert_{ \mc H } } $ and its domain 
$ 
	\text{Dom}( \barr{ \mc B\vert_{ \mc H } } ) 
$.

To begin, we define an operator $\hat{ \mc L} \colon C[ 0, 1 ] \cap C^2( 0, 1 ) \to C( 0, 1 ) $ by
$$
	\hat{ \mc L} f ( x )
		=
			x ( 1 - x )
			f''( x )
		-	\alpha
			f'( x )
		.
$$

\noindent
We will write $ \hat{ \mc L} f \in C[ 0, 1 ] $ whenever $\hat{ \mc L} f $ can be continuously extended to $ [ 0, 1 ] $.  Recalling the definition of $\mc L$ and $\mc D$ from Theorem \ref{thm main}, we see that $\mc L$ is the restriction of $\hat{ \mc L} $ to $\mc D$. 
We also define functions $ m \colon ( 0, 1 ] \to \mbb R $ and $ s \colon ( 0, 1 ] \to \mbb R $ by
$$
	m( x )
		= 	\int_1^x
				t^{ - 1 - \alpha } 
				( 1 - t )^{ \alpha - 1 }
				\, dt
		=
			-	\alpha^{ -1 }
				x^{ -\alpha } 
				( 1 - x )^{ \alpha } 
$$
and
$$
	s( x )
		= 	\int_1^x
				t^{ \alpha } 
				( 1 - t )^{ - \alpha } 
				\, dt 
		.
$$

\noindent
Note that $\hat{ \mc L} $ admits the factorization
$$
	\hat{\mc L}
	f
		=
			\frac{ 1 }{ m' }
			\left(
				\frac{ f' }{ s' }
			\right)'
		,
$$

\noindent
from which we obtain the formula
\begin{equation}
	f(x) - f(c)
		=
    			\frac{ f'(c) }{ s'(c) }
    			( s( x ) - s (c) )
    		+ 	\int_c^x
    				\int_c^y
    					\hat{\mc L} f(z)
    					m'( z )
    					dz
    				\,
    				s'(y)
    				dy
		,
			\quad
			x, c \in ( 0, 1 )
		.
	\label{generalized FTC}
\end{equation}

\noindent
Another identity that will be useful is
\begin{equation}
		\int_1^y
			m'( z )
			dz
		\,
		s'(y)
    		= 
        			m( y )
        			\,
        			s'(y)
    		= 
            		-\alpha^{ -1 }
			,
				\qquad
				y \in ( 0, 1 )
			.
	\label{integral identity for dominated convergence}
\end{equation}

\begin{prop}
    \label{description of generator and domain}
    The identity
    $ 
    	\barr{ \mc B\vert_{ \mc H } } 
    		=
            	\mc L
    $
holds,
    where $\mc L$ is as defined in Theorem \ref{thm main}.
\end{prop}

\begin{proof}

	We begin by showing that the following holds:
    \begin{equation}
    	f(x) - f(1)
    		=
    			\int_1^x
    				\int_1^y
    					\mc Lf(z)
    					m'( z )
    					dz
    				\,
    				s'(y)
    				dy
    		,
    			\quad
				f \in \mc D
			,
				\,
    			x \in [ 0, 1 ]
			.
		\label{generalized FTC on D}
    \end{equation}

	\noindent
	To do this, we will take limits in (\ref{generalized FTC}).
	First we take the limit $ c \to 1 $.
	The term $ \frac{f'(c)}{s'(c)} $ converges to zero due to \ref{boundary condition at one} (see Theorem \ref{thm main}).
	The limit of the integral is handled by the dominated convergence theorem -- a suitable bound follows from the boundedness of $ \mc L f $ and (\ref{integral identity for dominated convergence}).
	This establishes the formula for $ x \in ( 0, 1 ) $. 
	Taking now the limit $ x \to 0 $ (the dominated convergence theorem can be applied as before) establishes the $ x = 0 $ case.
	The $ x = 1 $ case is trivial.

	Now we show that 
	$
		\text{Dom}( \barr{ \mc B\vert_{ \mc H } } )
			\subset
				\mc D 
			.
	$
	Fixing $ f \in \text{Dom}( \barr{ \mc B\vert_{ \mc H } } ) $, there exists a sequence $ \{ f_n \}_{ n \ge 1 } $ of functions in $ \mc H $ such that
    \begin{equation}
    	f_n 
    		\longrightarrow 
    			f
    		\qquad
    			\text{and}
    		\qquad
    	\mc B
    	f_n
    		\longrightarrow 
    			\barr{ \mc B \vert_{ \mc H } }
    			f
			.
		\label{core approximation of domain}
    \end{equation}

    \noindent
    Noting that $ f_n \in \mc D $ for all $ n $, we can apply (\ref{generalized FTC on D}). 
    In this case, the identity $ \mc B f_n = \mc L f_n $ yields
    \begin{equation}
    	f_n(x) - f_n(1)
    		=
    			\int_1^x
    				\int_1^y
    					\mc B f_n(z)
    					m'( z )
    					dz
    				\,
    				s'(y)
    				dy
    		,
    			\quad
    			x \in [ 0, 1 ]
    		.
			\label{generalized FTC on core}
    \end{equation}
    
    \noindent
	Using (\ref{core approximation of domain}) and the dominated convergence theorem, we can take the limit $ n \to \infty $ above.
	A suitable bound follows from the boundedness of the sequence $ \{ \mc B f_n \} $ and (\ref{integral identity for dominated convergence}).
	We obtain 	
	\begin{equation}
    	f(x) - f(1)
    		=
    			\int_1^x
    				\int_1^y
    					\barr{ \mc B \vert_{ \mc H } }
    					f(z)
    					m'( z )
    					dz
    				\,
    				s'(y)
    				dy
    		,
    			\quad
    			x \in [ 0, 1 ]
    		.
		\label{generalized FTC on domain}
    \end{equation}

    \noindent
    Together with the fact that 
    $ 
        \barr{ \mc B \vert_{ \mc H } } 
        f
    		\in 
    			C( 0, 1 ) 
    		,
    $ 
    $ 
        m
    		\in 
    			C^1( 0, 1 ) 
    $ 
    and 
    $ 
    	s
    		\in 
    			C^2( 0, 1 ) 
    		,
    $ 
    this expression implies that $ f \in C^2( 0, 1 ) $.
	Differentiating the expression yields the identity
    \begin{equation}
    	\barr{ \mc B \vert_{ \mc H } } 
    	f
    		=
    			\frac{ 1 }{ m' }
    			\left(
    				\frac{ f' }{ s' }
    			\right)'
    		=
				\hat{\mc L}
				f
			\quad
				\text{ on }
				( 0, 1 )
			.
		\label{generator and generalized diff operator agree on domain}
    \end{equation}

    \noindent
    This shows that $ f $ satisfies \ref{continuous image under generalized diff operator}. 
    To obtain \ref{boundary condition at zero}, we recall that
    $$
    	\int_0^1 
    		( f_n(x) - f_n(0) )
    		x^{ - \alpha - 1 } 
    		(1-x)^{ \alpha - 1 }
    		\,
    		dx
    			=
    				0
    $$

    \noindent
    for all $ n $ and extend this to $ f $ by taking the limit $ n \to \infty $. 
    Once again, we apply the dominated convergence theorem.
    A preliminary bound can be obtained from (\ref{integral identity for dominated convergence}) and (\ref{generalized FTC on core}):
    \begin{align*}
		\left|
		x^{ -1 } 
		( f_n(x) - f_n(0) )
		\right|
    		& =
    			x^{ -1 }
        		\left|
				\int_0^x
    				\int_1^y
    					\mc B f_n(z)
    					m'( z )
    					dz
    				\,
    				s'(y)
    				dy
        		\right|
			\\
			& \le
    			x^{ -1 }
        		\norm{ \mc B f_n }
				\int_0^x
    				\int_y^1
    					m'( z )
    					dz
    				\,
    				s'(y)
    				dy			
			\\
			& =
        		\norm{ \mc B f_n }
				\alpha^{ -1 }
			.
    \end{align*}

    \noindent
	The boundedness of the sequence $ \{ \mc B f_n \} $ then provides a suitable bound.

    To obtain \ref{boundary condition at one}, we differentiate (\ref{generalized FTC on domain}) and compute
    \begin{align*}
    	\left|
		\frac{ 
				f'(x) 
			}{ 
				s'(x) 
			}
		\right|
        		& =
        				\left|
						\int_1^x
        					\barr{ \mc B \vert_{ \mc H } }
        					f(z)
        					m'( z )
        					dz
        				\,
						\right|
				\\
        		& \le
        				\norm{ \barr{ \mc B \vert_{ \mc H } } f }
						\int_x^1	
        					m'( z )
        					dz
        				\,
				\\
        		& =
        				\norm{ \barr{ \mc B \vert_{ \mc H } } f }
						( -m(x) )
				\\
        		& \xrightarrow[ x \to 1 ]{}
        				0
				.
    \end{align*}

    We have shown that 
    $
    	\text{Dom}( \barr{ \mc B\vert_{ \mc H } } )
			\subset
				\mc D 
	$ 
	and
    $ 
    	\barr{ \mc B\vert_{ \mc H } } 
    		=
            	\mc L
    $
    on $ \text{Dom}( \barr{ \mc B\vert_{ \mc H } } ) $ (see (\ref{generator and generalized diff operator agree on domain})). 
	Therefore, it only remains to show that 
	$ 
		\text{Dom}( \barr{ \mc B\vert_{ \mc H } } ) 
			= 
				\mc D 
			.
	$
    From Lemma 19.12 in \cite{kallenbergBook}, it suffices to show that $ \mc L $ satisfies the positive maximum principle.
    To this end, suppose that $ f \in \mc D $ has a nonnegative maximum at $ y \in [ 0, 1 ] $.
	If $ y \neq 1 $, then the desired inequality can be obtained as in Proposition \ref{existence of the generator}.
    If $ y = 1 $, we use 
    	\ref{continuous image under generalized diff operator}, 
		L'H\^opital's rule, 
		\ref{boundary condition at one}, 
		and 
		(\ref{integral identity for dominated convergence})
		to establish the existence of limits
	\begin{align*}
		\mc L
		f
		(1)
			& = 
					\lim_{ x \to 1 }
						\mc L f ( x )
			\\
			& = 
					\lim_{ x \to 1 }
            			\frac{ 1 }{ m'(x) }
            			\left(
            				\frac{ f' }{ s' }
            			\right)'
							\big(
									x
							\big)
			\\
			& = 
					\lim_{ x \to 1 }
            			\frac{ 1 }{ m(x) }
        				\frac{ f'(x) }{ s'(x) }
			\\
			& = 
					\lim_{ x \to 1 }
        				-\alpha 
						f'(x)
			\\
			& = 
					\lim_{ x \to 1 }
        				-\alpha 
        				\,
						\frac{ f(x) - f(1) }{ x - 1 }
			\\
			& = 
    				-\alpha 
					f'(1)
			.
	\end{align*}

	\noindent
	Noticing that $ f'(1) \ge 0 $ concludes the proof.
\end{proof}

\section{Generator Relations}
\label{section intertwining leftmost column with limit}

In this section, we show that our generators satisfy the commutation relations appearing in Theorem \ref{markovConvergenceFromIntertwining}.
Here, we rely on an alternative description of the limiting generator in terms of Bernstein polynomials.

For $ k \ge 0 $, let $ \mc P_k $ be the subspace of $ \mc P $ consisting of polynomials with degree at most $ k $.
Similarly, define
$$
	\mc H_k
		=
			\mc H \cap \mc P_k
		,
			\qquad
			k \ge 0
		.
$$

\noindent
Recall the Bernstein polynomials
$$
	b_{ i, k } ( x )
		=
			\binom{ k }{ i }
			x^i
			( 1 - x )^{ k - i }
		,
			\quad
			i \in \mbb Z
		,
			\,
			k \ge 0
		.
$$

\noindent
Note that $ b_{ i, k } \equiv 0 $ whenever $ i < 0 $ or $ i > k $.
For each $ k \ge 0 $, the collection 
$
	\{
		b_{ i, k }
	\}_{ i = 0 }^k
$
forms a basis of $ \mc P_k $
and a partition of unity -- that is,
$
	\sum_{ i = 0 }^k
		b_{ i, k }
			\equiv
				1
			.
$
We also have the relations
\begin{align}
	b_{ i, k }'
		& =
			k
			(
				b_{ i - 1, k - 1 }
			-	b_{ i, k - 1 }
			)
		,
		\label{bernsteinDerivative}
		\\
	b_{ i, k }
		& =
    			\tfrac{ k + 1 - i }{ k + 1 }
    			\,
    			b_{ i, k + 1 }
    		+	\tfrac{ i + 1 }{ k + 1 }
    			\,
    			b_{ i + 1, k + 1 }
		,
		\label{bernsteinExpansion}
\end{align}

\noindent
and
\begin{equation}
		\label{bernsteinScaling}
	x ( 1 - x )
	\,
	b_{ i, k }
		=
			\tfrac{
					( i + 1 )
					( k + 1 - i )
				}{
					( k + 1 )
					( k + 2 )
				}
			\,
			b_{ i + 1, k + 2 }
    	,
\end{equation}

\noindent
which hold whenever the relevant quantities are defined.

For $ n \ge 1 $, we define a transition kernel from $ [ 0, 1 ] $ to $ [ n ] $ by
$$
	\bernsteinKernel_n( x, i )
		=
			b_{ i, n }( x )
		+	\leftMostColDistSym_n( i )
			b_{ 0, n }( x )
		.
$$

\begin{prop}
	\label{bernsteinKernelAndFiltration}

Let $ n \ge 1 $. As an operator from $ C( [ n ] ) $ to $ C[ 0, 1 ] $, $ \bernsteinKernel_n $ is injective and
\begin{align}	
		\mc H_n
			& = 
    				\Bigg
    				\{
    					\sum_{ j = 0 }^n
    						c_j
    						b_{ j, n }
    				:	c_0, \ldots, c_n \in \mbb R,
    					\,\,\,
						c_0
					=
    					\sum_{ j = 1 }^n
    						\leftMostColDistSym_n( j ) 
    						c_j
    				\Bigg
    				\}
        \label{HnAsDiscreteKernel}
			\\
			& =
    				\text{range } \bernsteinKernel_n
		\label{rangeOfBernsteinKernel}
			.
\end{align}

\end{prop}

\begin{proof}

	Let $ n \ge 1 $.
	From the independence of the Bernstein polynomials and the identity
	$$
		\text{range } 
		\bernsteinKernel_n
			=
				\text{span}
						\big
						\{
							b_{ i, n }( x )
                		+	\leftMostColDistSym_n( i )
                			b_{ 0, n }( x )
						\big
						\}_{ i =  1 }^n
			,
	$$

	\noindent
	it follows that the range of $ \bernsteinKernel_n $ is an $n$-dimensional space.
	As a result, $ \bernsteinKernel_n $ is injective.
	Observing that the right hand side of (\ref{HnAsDiscreteKernel}) has dimension at most $ n $ and contains the range of $ \bernsteinKernel_n $,
it follows that these two spaces are equal.
	Since $ \mc H_n $ also has dimension $ n $ (see Proposition \ref{propertiesOfB}\ref{eigenbasisOfH}), it only remains to show that the range of $ \bernsteinKernel_n $ is contained in $ \mc H_n $.
	The containment in $ \mc P_n $ is clear. 
	For the containment in $ \mc H $, we simply compute, for $ i \in [ n ] $,
	\begin{align*}
		& \hspace{-6mm}
		\eta 
    		(
				b_{ i, n }( x )
    		+	\leftMostColDistSym_n( i )
    			b_{ 0, n }( x )
    		)
		\\
			& \hspace{ 8 mm } = 
        				\binom{ n }{ i }
        	           	\int_0^1
        					x^{ i - \alpha - 1 }
        					( 1 - x )^{ n - i + \alpha - 1 }
                       		\,
                       		dx
        			- 	n 
            			\alpha^{ -1 }
        				\leftMostColDistSym_n( i )
                        \int_0^1 
                        	x^{ - \alpha } 
                        	( 1 - x )^{ n - 1 + \alpha }
            				\,
                        	dx
			\\
			& \hspace{ 8 mm } = 
        				\binom{ n }{ i }
        				\frac	{
        							\Gamma( i - \alpha )
        							\Gamma( n - i + \alpha )
            					}{
        							\Gamma( n )
            					}
    				-	n 
        				\alpha^{-1}
    					\leftMostColDistSym_n( i )
        				\frac	{
        							\Gamma( 1 - \alpha )
        							\Gamma( n + \alpha )
            					}{
        							\Gamma( n + 1 )
            					}
			\\
			& \hspace{ 8 mm } = 
						0
					.
					\qedhere
	\end{align*}
\end{proof}

\begin{prop}
    \label{generatorOnBernstein}

    The action of $ \mc B $ on the Bernstein polynomials is given by
    $$
    	\mc B  
    	b_{ i, n }
    		=
           		n ( n + 1 )
				\sum_{ k = 0 }^n
					(
						\leftMostColLocalProb_{ k, i }
					-	\indicator( k = i )
					)
					\,
           			b_{ k, n }
    		,
    		\qquad
			0 \le i \le n
			.		
    $$
    
\end{prop}

\begin{proof}

    Let $ n \ge 2 $ and 
    $
    	0
			\le
				i
			\le
				n
			.
    $ 
    Applying (\ref{bernsteinDerivative}) twice, we see that 
    \begin{align*}
    	b_{ i, n }''
    		& =
        			n
        			(
        				b_{ i - 1, n - 1 }'
        			-	b_{ i, n - 1 }'
        			)
			\\
    		& =
        			n
        			( n - 1 )
					(
        				b_{ i - 2, n - 2 }
        			-	2 b_{ i - 1, n - 2 }
        			+	b_{ i, n - 2 }
        			)
			.
    \end{align*}

    \noindent
    Applying now (\ref{bernsteinScaling}), we have that
    \begin{align}
    \begin{split}
			\label{scaledBernsteinSecondDerivativeExpansion}
    	&
		x ( 1 - x )
		b_{ i, n }''( x )
    	\\
			& \hspace{ 6.5 mm } =
        			n
        			( n - 1 )
					\Big(
            			\tfrac{
            					( i - 1 )
            					( n + 1 - i )
            				}{
            					( n - 1 )
            					n
            				}
            			\,
            			b_{ i - 1, n }( x )
					-	\tfrac{
            					2 i
            					( n - i)
            				}{
            					( n - 1 )
            					n
            				}
            			\,
            			b_{ i, n }( x )
					+	\tfrac{
            					( i + 1 )
            					( n - 1 - i )
            				}{
            					( n - 1 )
            					n
            				}
            			\,
            			b_{ i + 1, n }( x )
        			\Big)
			\\
			& \hspace{ 6.5 mm } =
            			( i - 1 )
            			( n + 1 - i )
            			\,
            			b_{ i - 1, n }( x )
					-	2 i
            			( n - i )
            			\,
            			b_{ i, n }( x )
					+	( i + 1 )
            			( n - 1 - i )
            			\,
            			b_{ i + 1, n }( x )
	\end{split}
    \end{align}

    \noindent
	Using (\ref{bernsteinDerivative}) and (\ref{bernsteinExpansion}), we find that
    \begin{align}
    \begin{split}
			\label{bernsteinFirstDerivativeExpansion}
		b_{ i, n }'
    		& =
					n
        			(
        				b_{ i - 1, n - 1 }
        			-	b_{ i, n - 1 }
        			)
			\\
    		& =
					n
        			\Big(
            			\tfrac{ n + 1 - i }{ n }
            			\,
            			b_{ i - 1, n }
            		+	\tfrac{ i }{ n }
            			\,
            			b_{ i, n }
					-	\tfrac{ n - i }{ n }
            			\,
            			b_{ i, n }
			 		-	\tfrac{ i + 1 }{ n }
            			\,
            			b_{ i + 1, n }
        			\Big)
			\\
    		& =
            			( n + 1 - i )
            			\,
            			b_{ i - 1, n }
            		+	( 2 i - n )
            			\,
            			b_{ i, n }
			 		-	( i + 1 )
            			\,
            			b_{ i + 1, n }
			.
	\end{split}
    \end{align}

    \noindent
    As a result,
    \begin{align*}
		\mc B 
		b_{ i, n }
			& = 
            			( i - 1 - \alpha )
            			( n + 1 - i )
            			\,
            			b_{ i - 1, n }
					-	(
							\alpha
        					( 2 i - n )
    					+	2 i
                			( n - i )
						)
            			\,
            			b_{ i, n }
			\\
			& \quad		
					+	( i + 1 )
            			( n - 1 - i + \alpha )
            			\,
            			b_{ i + 1, n }
			\\
			& = 
            		n ( n + 1 )
					\left(
						\leftMostColLocalProb_{ i - 1, i }
						\,
            			b_{ i - 1, n }
    				+	(
							\leftMostColLocalProb_{ i, i }
						-
							1
						)
						\,
            			b_{ i, n }	
					+	\leftMostColLocalProb_{ i + 1, i }
						\,
            			b_{ i + 1, n }
					\right)
			\\
			& = 
            		n ( n + 1 )
					\sum_{ k = i - 1 }^{ i + 1 }
						(
							\leftMostColLocalProb_{ k, i }
						-	\indicator( k = i )
						)
						\,
            			b_{ k, n }
			.
    \end{align*}

    \noindent
    Recalling that 
	$
    		\leftMostColLocalProb_{ k, i }
    	-	\indicator( k = i )
	$
	is zero unless $ i - 1 \le k \le i + 1 $ and $ b_{ k, n } \equiv 0 $ unless $ 0 \le k \le n $, we can change the lower and upper limits of the sum to $ 0 $ and $ n $, respectively.
    This establishes the $ n \ge 2 $ case. 
    When $ n = 1 $, we observe that (\ref{bernsteinFirstDerivativeExpansion}) still holds and the first and last quantities of
    (\ref{scaledBernsteinSecondDerivativeExpansion}) 
    are still equal.
	When $ n = 0 $, the claim is trivial.
\end{proof}

\begin{prop}
\label{leftMostChainAndLimitIntertwining}

    For $ n \ge 1 $, the following relation holds on $ C([ n ])$:
    $$
    	\mc B
    	\bernsteinKernel_n
    		=
    			\bernsteinKernel_n 
    			\,
    			n ( n + 1 )
    			( 
    				Q_n^{ ( \alpha, 0 ) } 
    			- 	\mb 1 
    			) 
    		.
    $$
\end{prop}

\begin{proof}

	Let $ n \ge 1 $ and $ i \in [ n ] $.
	Define $ e_i \colon [ n ] \to \mbb R $ by 
	$
		e_i
			= 
				\indicator( i = \cdot )
			.
	$
	From Proposition \ref{generatorOnBernstein}, we have that
	\begin{align*}
		&
		n^{ -1 } 
		( n + 1 )^{ -1 }
		\mc B
		\bernsteinKernel_n
		e_i
		\\
			& \qquad = 
        		n^{ -1 } 
        		( n + 1 )^{ -1 }
				\mc B
				(
        			b_{ i, n }
        		+	b_{ 0, n }
        			\leftMostColDistSym_n( i )
				)
			\\
			& \qquad = 
				\sum_{ k = 0 }^n
					(
						\leftMostColLocalProb_{ k, i }
					-	\indicator( k = i )
					+	\leftMostColDistSym_n( i )
    					(
    						\leftMostColLocalProb_{ k, 0 }
    					-	\indicator( k = 0 )
    					)
					)
					\,
           			b_{ k, n }
			\\
			& \qquad = 
					(
						\leftMostColLocalProb_{ 0, i }
					+	\leftMostColDistSym_n( i )
    					(
    						\leftMostColLocalProb_{ 0, 0 }
    					-	1
    					)
					)
					\,
           			b_{ 0, n }
			+	\sum_{ k = 1 }^n
					(
						\leftMostColLocalProb_{ k, i }
					-	\indicator( k = i )
					+	\leftMostColDistSym_n( i )
						\leftMostColLocalProb_{ 1, 0 }
						\indicator( k = 1 )
					)
					\,
           			b_{ k, n }
			.
	\end{align*}

	\noindent
	On the other hand, Proposition \ref{propIntertwining} gives us that
    \begin{align*}
		\bernsteinKernel_n
    	( 
    		Q_n^{ ( \alpha, 0 ) } 
    	- 	\mb 1 
    	) 
    	e_i
			& = 
    			\sum_{ k = 1 }^n
    				(
            			b_{ k, n }
            		+	b_{ 0, n }
            			\leftMostColDistSym_n( k )
    				)
                	(
    				( 
                		Q_n^{ ( \alpha, 0 ) } 
                	- 	\mb 1 
                	) 
                	e_i
                	)( k )	
    		\\
			& = 
    			\sum_{ k = 1 }^n
    				(
            			b_{ k, n }
            		+	b_{ 0, n }
            			\leftMostColDistSym_n( k )
    				)
    				( 
                		Q_n^{ ( \alpha, 0 ) } 
                	- 	\mb 1 
                	)( k, i ) 
    		\\
			& = 
    			\sum_{ k = 1 }^n
    				(
            			b_{ k, n }
            		+	b_{ 0, n }
            			\leftMostColDistSym_n( k )
    				)
    				( 
                		\leftMostColLocalProb_{ k, i }
                	- 	\indicator( i = k )
    				+ 	\leftMostColDistSym_n( i )
						\leftMostColLocalProb_{ 1, 0 }
    					\indicator( k = 1 ) 
    				)
    		.
    \end{align*}

	\noindent
	To show that the two expressions are equal, it will suffice to show that the coefficients of $ b_{ k, n } $ are the same in each.
	For $ k \ge 1 $, this is immediate. 
	For $ k = 0 $, we observe that each of the above functions lies in $ \mc H_n $ (see Proposition \ref{propertiesOfB} and (\ref{rangeOfBernsteinKernel})) and apply (\ref{HnAsDiscreteKernel}).
\end{proof}

\section{The Convergence Argument}
\label{section convergence argument}

In this section, we verify the convergence condition appearing in Theorem \ref{markovConvergenceFromIntertwining}.
We rely on a description of the inverse of the transition operator $ \bernsteinKernel_n $ in terms of a variant of the Bernstein polynomials.

These variants fall into the class of degenerate Bernstein polynomials \cite{kimKim2018} and are given by
$$
	b_{ i, k, n }^* ( x )
		=
			\binom{ k }{ i }
			\frac{
    			( n x )^{ \downarrow i }
    			( n - n x )^{ \downarrow ( k - i ) }
			}{
    			n^{ \downarrow k }
			}
		,
			\quad
    			0 \le i \le k \le n
		.
$$

\begin{prop}
        \label{bernsteinExtendedPieriRule}
        
        For $ k \ge i \ge 0 $, we have the expansions
        $$
        	b_{ i, k }
        		=
        			\sum_{ j = 0 }^n
        				b_{ i, k, n }^*
        					\big( 
        						\tfrac{ j }{ n } 
        					\big)
        				b_{ j, n }
        		,
        		\quad
        			n \ge k
        		.
        $$

\end{prop}

\begin{proof}
	The expansions of a Bernstein polynomial in the Bernstein bases are given in Equation (2) in \cite{rababah2004}.
	Let us verify that the coefficients in those expansions match the coefficients in the above expansions.
	Fix $ n \ge k \ge i \ge 0 $. 
	The coefficient of $ b_{ j, n } $ in the above expansion is given by
$$
	b_{ i, k, n }^* \left( \frac{ j }{ n }  \right)
		=
			\binom{ k }{ i }
			\frac{
    			j^{ \downarrow i }
    			( n - j )^{ \downarrow ( k - i ) }
			}{
    			n^{ \downarrow k }
			}
		.
$$
	
	\noindent
	When $ j < i $ or $ j > n - k + i $, it is clear that this coefficient is zero.
	If instead $ i \le j \le n - k + i $, this coefficient is reduces to
	\begin{align*}
		\binom{ k }{ i }
		\frac{
   			j^{ \downarrow i }
   			( n - j )^{ \downarrow ( k - i ) }
		}{
   			n^{ \downarrow k }
		}
			& = 
            		\binom{ k }{ i }
            		\frac{
               			\frac{ j! }{ ( j - i )! }
               			\frac{ ( n - j )! }{ ( n - j - k + i )! }
            		}{
               			\frac{ n! }{ ( n - k )! }
            		}
			\\
			& = 
            		\binom{ k }{ i }
            		\frac{
               			\frac{ ( n - k )! }{ ( j - i )! ( n - j - k + i )! }
            		}{
               			\frac{ n! }{ j! ( n - j )! }
            		}
			\\
			& = 
            		\binom{ k }{ i }
            		\frac{
               			\binom{ n - k }{ j - i }
            		}{
               			\binom{ n }{ j }
            		}
			.
	\end{align*}
	
	\noindent
	In either case, this coefficient agrees with the coefficient in \cite{rababah2004}.
\end{proof}

Let $ \iota_n \colon [ n ] \to [ 0 , 1 ] $ be defined by $ j \mapsto \frac{ j }{ n } $ and $ \rho_n \colon C[ 0, 1 ] \to C[ n ] $ be the associated projection, 
$ 
	f 
		\mapsto 
			f \circ \iota_n 
	.
$

\begin{prop}
	\label{preimagesOfBernsteinBasisOfH}

	For $ n \ge k \ge i \ge 1 $, we have the identity
    $$
    	\bernsteinKernel_n 
    	\rho_n
    	(
    		b_{ i, k, n }^*
    	+	\leftMostColDistSym_k( i )
    		b_{ 0, k, n }^*
    	)
    		=
        			b_{ i, k }
            	+	\leftMostColDistSym_k( i )
            		b_{ 0, k }
    		.
    $$

\end{prop}

\begin{proof}

	It follows from definition that
	\begin{align*}
		\bernsteinKernel_n 
    	\rho_n
    	(
    		b_{ i, k, n }^*
    	+	\leftMostColDistSym_k( i )
    		b_{ 0, k, n }^*
    	)
			& = 
    			\sum_{ j = 1 }^n
    				(
    					b_{ j, n }
    				+ 	\leftMostColDistSym_n( j )
    					b_{ 0, n }
    				)
                	\big(
                		b_{ i, k, n }^*
        					\big( 
        						\tfrac{ j }{ n } 
        					\big)
                	+	\leftMostColDistSym_k( i )
                		b_{ 0, k, n }^*
        					\big( 
        						\tfrac{ j }{ n } 
        					\big)
                	\big)	
    		.
	\end{align*}

	\noindent
	Meanwhile, Proposition \ref{bernsteinExtendedPieriRule} gives us the expansion
    $$
        	b_{ i, k }
    	+	\leftMostColDistSym_k( i )
    		b_{ 0, k }
        		=
        			\sum_{ j = 0 }^n
                    	\big(
                    		b_{ i, k, n }^*
            					\big( 
            						\tfrac{ j }{ n } 
            					\big)
                    	+	\leftMostColDistSym_k( i )
                    		b_{ 0, k, n }^*
            					\big( 
            						\tfrac{ j }{ n } 
            					\big)
                    	\big)	
        				b_{ j, n }
        		.
    $$
    
    \noindent
	Upon comparison, we find that the coefficient of $ b_{ j, n } $ is the same in both expressions whenever $ j \ge 1 $. 
    Since both functions lie in $ \mc H_n $, the coefficients of $ b_{ 0, n } $ must agree as well (see (\ref{HnAsDiscreteKernel})). 
    As a result, the two functions are equal.
\end{proof}

\begin{prop}
	\label{degenerateBernsteinConvergence}

	For $ k \ge i \ge 0 $, we have the convergence
	$$
		b_{ i, k, n }^*
			\xrightarrow[ n \to \infty ]{}
				b_{ i, k }
			.
	$$
	
\end{prop}

\begin{proof}

	We write
	\begin{align*}
    	b_{ i, k, n }^* ( x )
    		& =
    			\binom{ k }{ i }
    			\frac{ 1 }{ n^{ \downarrow k } }
    			\prod_{ r = 0 }^{ i - 1 }
					( n x - r )
				\prod_{ s = 0 }^{ k - i - 1 }
	    			( n - n x - s )
			\\
    		& =
    			\binom{ k }{ i }
    			\frac{ n^k }{ n^{ \downarrow k } }
    			\prod_{ r = 0 }^{ i - 1 }
					\left( 
						x - \frac{ r }{ n } 
					\right)
				\prod_{ s = 0 }^{ k - i - 1 }
	    			\left( 
							1 - x - \frac{ s }{ n } 
					\right)
			,
	\end{align*}
	
	\noindent
	and handle each factor separately. 
	The constants $ \frac{ n^k }{ n^{ \downarrow k } } $ converge to $ 1 $ and each factor in a product converges to either $ u( x ) = x $ or $ v( x ) = 1 - x $.
\end{proof}

\begin{prop}
	\label{intertwiningOperatorAndProjectionConvergence}

	Let $ f \in \mc H $ and
	fix $ m \ge 1 $ such that $ f \in \mc H_m $.
	Then we have the convergence
	$$
    	(
			\bernsteinKernel_n^{ -1 }
		-	\rho_n
		)
		f
				\xrightarrow[ n \to \infty ]{ n \ge m }
					0
	$$

	\noindent
	in the sense of Definition \ref{convergenceInDifferentSpaces}.
\end{prop}

\begin{proof}
	
	It suffices to consider the case when
	$ 
		f = 
            	b_{ i, k }
        	+	\leftMostColDistSym_k( i )
        		b_{ 0, k }
	$
	for some $ i $ and $ k $ satisfying $ 1 \le i \le k $.
	Defining
	$
		f_n
			=
        		b_{ i, k, n }^*
        	+	\leftMostColDistSym_k( i )
        		b_{ 0, k, n }^*
	$
	for $ n \ge 1 $,
	it follows from Proposition \ref{preimagesOfBernsteinBasisOfH} that
	\begin{align*}
    	(
			\bernsteinKernel_n^{ -1 }
		-	\rho_n
		)
		f
				& =
                    	\rho_n
        				(
    						f_n
    					-	f
    					)
				.
	\end{align*}
		
	\noindent
	Since the $ \rho_n $ are uniformly bounded, the result follows from Proposition \ref{degenerateBernsteinConvergence}.
\end{proof}

\section{Semigroup Relations from Generator Relations}
\label{section semigroup relation from generator relation}

In this section, we provide general conditions under which commutation relations involving generators lead to the corresponding relations for their semigroups.

\begin{theorem}

    Let $ \generalGenerator $ and $ \generalGeneratorTwo $ be the generators of the Feller semigroups $ \generalSemigroup $ and $ \generalSemigroupTwo $, respectively, and let $ \generalDomain $ and $ \generalDomainTwo $ denote their respective domains.
    Suppose that there is a subspace
    $ 
    	\generalSubspaceOfDomain 
    		\subset 
    			\generalDomain 
	$,
    a linear operator 
    $ 
    	L \colon \, \barr{ \! \generalSubspaceOfDomain } \to \barr{ \generalDomainTwo }
	$, 
	and a set 
	$ 
		\setOfLambdaValues 
			\subset 
				( 0, \infty ) 
	$ 
	such that
    \begin{enumerate}[ label = (\roman*) ]
    
		\item
		$ L $ is bounded,

		\item
		$ \setOfLambdaValues $ is unbounded,

    	\item
    	$ 
			\generalSubspaceOfDomain 
    			\subset 
    				( \lambda - \generalGenerator ) \generalSubspaceOfDomain 
		$ 
		for $ \lambda \in \setOfLambdaValues $, and

    	\item
    	$ 
			L
			\generalGenerator
				= 
					\generalGeneratorTwo
					L
		$ 
		on $ \generalSubspaceOfDomain $.
    
    \end{enumerate}

    \noindent
    Then 
	$
		L
		\generalSemigroup 
			= 
				\generalSemigroupTwo
				L
	$
	on $ \, \barr{ \! \generalSubspaceOfDomain } $ for each $ t \ge 0 $.

	\label{thmSemigroupRelationFromGenerator}
\end{theorem}

\begin{proof}

    Fix $ \lambda \in \setOfLambdaValues $ 
    and let $ R_\lambda^{ \generalGenerator } $ and $ R_\lambda^{ \generalGeneratorTwo } $ 
    be the resolvent operators corresponding to $ \generalGenerator $ and $ \generalGeneratorTwo $ respectively. 
    It follows from (iii) that $ \generalSubspaceOfDomain $ is invariant under $ R_{ \lambda }^{ \generalGenerator } $.
    Combining this with (iv), we obtain the following relation on $ \generalSubspaceOfDomain $:
    \begin{align*}
    	R_\lambda^{ \generalGeneratorTwo }
    	L
    		& =
            	R_\lambda^{ \generalGeneratorTwo }
            	L
            	( \lambda - \generalGenerator ) 
            	R_\lambda^{ \generalGenerator }
    		\\
    		& =
            	R_\lambda^{ \generalGeneratorTwo }
            	( \lambda - \generalGeneratorTwo )
            	L
    			R_\lambda^{ \generalGenerator }
    		\\
    		& =
            	L
    			R_\lambda^{ \generalGenerator }
			.
    \end{align*}

    \noindent
	It then follows easily that
    $$ 
    	L
		\lambda 
    	( 
    		\lambda 
    		R_\lambda^{ \generalGenerator } 
    	- 	I 
    	)
    		=
            	\lambda 
            	( 
            		\lambda 
            		R_\lambda^{ \generalGeneratorTwo } 
            	- 	I 
            	)
				L
			\quad
			\text{on } \generalSubspaceOfDomain
			,
    $$
    
    \noindent
	or equivalently, 
    $ 
    	L
		\generalGenerator_\lambda 
    		= 
				\generalGeneratorTwo_\lambda 
				L
    $
    on $ \generalSubspaceOfDomain $, where $ \generalGenerator_\lambda $ and $ \generalGeneratorTwo_\lambda $ are the Yosida approximations of $ \generalGenerator $ and $ \generalGeneratorTwo $ respectively.
	Noting that $ \generalSubspaceOfDomain $ is invariant under $ \generalGenerator_\lambda $, this extends to nonnegative integers $ k $:
	$$ 
    	L
		\generalGenerator_\lambda^k 
    		= 
    			\generalGeneratorTwo_\lambda^k
				L
			\quad
			\text{on } \generalSubspaceOfDomain
			.
	$$

    \noindent
	Applying now (i), we have for $ f \in \generalSubspaceOfDomain $ and $ t \ge 0 $ the identity
    \begin{align*}
		L
		e^{ t \generalGenerator_{ \lambda } }
		f
    		& =
    			L
				\sum_{ k = 0 }^\infty
    				\frac{ t^k }{ k! }
    				( \generalGenerator_\lambda^k f )
    		\\
    		& =
    			\sum_{ k = 0 }^\infty
    				\frac{ t^k }{ k! }
    				( L \generalGenerator_\lambda^k f )
    		\\
    		& =
    			\sum_{ k = 0 }^\infty
    				\frac{ t^k }{ k! }
    				( \generalGeneratorTwo_\lambda^k L f )
    		\\
    		& =
    			e^{ t \generalGeneratorTwo_{ \lambda } }
    			L
				f
    		.
    \end{align*}

	\noindent
	Letting $ \lambda $ become arbitrarily large (see (ii)) yields
	$
		L
		\generalSemigroup 
		f
			= 
				\generalSemigroupTwo
				L
				f
			.
    $ 
    This establishes the result on $ \generalSubspaceOfDomain $. 
	The extension to $ \, \barr{ \! \generalSubspaceOfDomain } $ follows from the boundedness of $ L $.
\end{proof}

\begin{corollary}

	\label{generator to semigroup relation with finite filtration}
    Let $ \generalGenerator $ and $ \generalGeneratorTwo $ be the generators of the Feller semigroups $ \generalSemigroup $ and $ \generalSemigroupTwo $, respectively, and let $ \generalDomain $ and $ \generalDomainTwo $ denote their respective domains.
    Suppose that there is a subspace
    $ 
    	\generalSubspaceOfDomain 
    		\subset 
    			\generalDomain 
	$,
    a linear operator 
    $ 
    	L \colon \generalSubspaceOfDomain \to \barr{ \! \generalDomainTwo }
	$, 
	and a filtration of $ \generalSubspaceOfDomain $ by finite dimensional spaces
	$ 
		\{ \generalSubspaceOfDomain_k \}_{ k \ge 1 }
	$ 
	such that
    \begin{enumerate}[ label = (\roman*) ]
    
        \item	
    	$ 
			\generalGenerator
			\generalSubspaceOfDomain_k
    			\subset 
    				\generalSubspaceOfDomain_k
		$
		for all $ k $, and

    	\item
    	$ 
			L
			\generalGenerator
				= 
					\generalGeneratorTwo
					L
		$ 
		on $ \generalSubspaceOfDomain $.
    
    \end{enumerate}

    \noindent
    Then 
	$
		L
		\generalSemigroup 
			= 
				\generalSemigroupTwo
				L
	$
	on $ \generalSubspaceOfDomain $ for each $ t \ge 0 $.

\end{corollary}

\begin{proof}

Let $ k \ge 1 $.
It follows from (i) that $ \generalSubspaceOfDomain_k $ is invariant under the injective operators 
\linebreak[4]
$ 
	\{
	\lambda - \generalGenerator 
	\}_{ \lambda > 0 }
$.
Together with the fact that $ \generalSubspaceOfDomain_k $ is finite-dimensional, this implies that
$$
	( \lambda - \generalGenerator )
	\generalSubspaceOfDomain_k
		=
			\generalSubspaceOfDomain_k
		,
		\quad
		\lambda > 0 
		.
$$

\noindent
Letting
$ 
	L_k
		\colon
			\generalSubspaceOfDomain_k \to \, \barr{ \! \generalDomainTwo }
$
denote the restriction of $ L $ to $ \generalSubspaceOfDomain_k $, it follows from (i) and (ii) that
$$ 
	L_k
	\generalGenerator
		= 
			\generalGeneratorTwo
			L_k
		\quad
		\text{on }
		\generalSubspaceOfDomain_k
		.
$$

\noindent
Since $ \generalSubspaceOfDomain_k $ is finite-dimensional, $ L_k $ is bounded and 
$ 
	\, \barr{ \! \generalSubspaceOfDomain_k }
		=
			\generalSubspaceOfDomain_k
		.
$
Applying Theorem \ref{thmSemigroupRelationFromGenerator}, we find that
$
	L
	\generalSemigroup 
		= 
			\generalSemigroupTwo
			L
$
on $ \generalSubspaceOfDomain_k $ for each $ t \ge 0 $.
Taking a union over $ k $ extends the identity to $ \generalSubspaceOfDomain $.
\end{proof}

\section{Proofs of Main Results}
\label{section proofs of main results}

\begin{proof}[Proof of Theorem \ref{full statement of all intertwining results}]

	The first claim was proved in Proposition \ref{propIntertwining}.
	For the second claim, we appeal to
	Corollary \ref{generator to semigroup relation with finite filtration}.
	We take
	$ \generalGenerator = n ( n + 1 ) ( \transKernelLeftMostCol_n^{ ( \alpha, 0 ) } - \mb 1 ) $,
	$ \generalGeneratorTwo = \mc L $, 
	$ L = \bernsteinKernel_n $, and
	$ \generalSubspaceOfDomain = C( [ n ] ) = \generalSubspaceOfDomain_k $ for all $ k $.
	The containment 
	$ 
		\generalGenerator
		\generalSubspaceOfDomain_k
			\subset 
				\generalSubspaceOfDomain_k
	$
	holds trivially and the identity 
	$ 
		L
		\generalGenerator
			= 
				\generalGeneratorTwo
				L
	$ 
	was established in Proposition \ref{leftMostChainAndLimitIntertwining}.
	Applying Corollary \ref{generator to semigroup relation with finite filtration}, we obtain the desired identity in terms of transition operators, which implies the same relation in terms of transition kernels.
\end{proof}

\begin{proof}[Proof of Theorem \ref{full statement of convergence result}]
The claim about the existence of initial distributions for $\upDownChainBiased_n$ follows from Theorem \ref{full statement of all intertwining results}.
	The second claim follows from applying Theorem \ref{markovConvergenceFromIntertwining} with 
	$ E = [ 0, 1 ] $,
	$ A = \mc L $,
	$ D = \mc H $,
	$ E_n = [ n ] $,
	$ Z_n = \leftMostColChain_n $,
	$ \gamma_n(j) = \frac{ j }{ n } $,
	$ D_n = \mc H_n $, 
	$ L_n = \bernsteinKernel_n^{ -1 } $,
	$ \delta_n^{ -1 } = n ( n + 1 ) $,
	and
	$ \eps_n^{ -1 } = n^2 $.
	To verify that 
    	$ A $ is the generator of a conservative Feller semigroup on $ C[ 0, 1 ] $, 
    	$ D $ is a core for $ A $, 
    	and 
		$ D $ is invariant under $ A $, 
	we appeal to Propositions 
		\ref{description of generator and domain}, 
		\ref{existence of the generator}, 
		and 
		\ref{propertiesOfB}. 
	Condition \ref{generator relation in general result} can be obtained from 
		the identity in Proposition \ref{leftMostChainAndLimitIntertwining} by recalling that 
		$ \bernsteinKernel_n $ is injective (see Proposition \ref{bernsteinKernelAndFiltration}) 
		and 
		that each $ f $ in $ D = \mc H $ lies in $ D_n = \mc H_n $ for large $ n $.
	Condition \ref{convergence condition in general result} is exactly the result of Proposition \ref{intertwiningOperatorAndProjectionConvergence}.
\end{proof}

\begin{proof}[Proof of Theorem \ref{thm main}]
    Define $\iota : \compSet \to \setOfOpenSets$ by 
    $$
    	\iota( \sigma ) 	
    		= 	
    				\bigg( 
    						0, 
    						\frac{ \sigma_1 }{ | \sigma | } 
    				\bigg ) 	
    			\cup 	
    				\bigg( 
    						\frac{ \sigma_1 }{ | \sigma | }, 
    						\frac{ \sigma_1 + \sigma_2 }{ | \sigma | } 
    				\bigg) 	
    			\cup 
    				\ldots 
    			\cup	
    				\bigg( 
    						\frac{ 
    								| \sigma | - \sigma_{ \ell( \sigma ) }
    							}{ 
    								| \sigma | 
    							}, 
    						1 
    				\bigg)
    	.
    $$
    From \cite[Theorem 1.3]{krdr2020}, we have that if
    \[ \iota(\upDownChain_n(0))  \Longrightarrow \upDownChain(0),\]
    then
    \[ \left(\iota(\upDownChain_n(\lfloor n^2t\rfloor ))\right)_{t\geq 0} \Longrightarrow \left(\upDownChain(t)\right)_{t\geq 0},\]
    where $\lfloor a \rfloor$ is the integer part of $a$ and the convergence is in distribution on the Skorokhod space $D([0,\infty),\setOfOpenSets)$, where the metric on $\setOfOpenSets$ is given by the Hausdorff distance between the complements (complements being taken in $[0,1]$).  If $\xi$ were continuous, the result would follow immediately, but $\xi$ is discontinuous.  However, it is straightforward to show that if $u_n\to u$ in $\setOfOpenSets$ and $\xi(u_n) \to c >0$, then $\xi(u)=c$.  
 
 Assuming now that $\upDownChainBiased_n$ is running in stationarity, the fact that $\iota(\upDownChainBiased_n(0))$ converges in distribution to an $(\alpha,0)$ Poisson-Dirichlet interval partition distribution follows from \cite{PitmWink09} and the fact that $\phi(\upDownChainBiased_n)$ is a Markov chain follows from Theorem \ref{full statement of all intertwining results}.  Observe that $(\upKernelBiased)^{n-1}((1), \cdot)$ is the stationary distribution of $\upDownChainBiased_n$ and, in the $(\alpha,0)$ ordered Chinese Restaurant Process growth step, no new table is ever created at the start of the list.  
    Thus, for every $k$, $\phi(\upDownChainBiased_n(k))$ is distributed like the size of the table containing $1$ in the usual $(\alpha,0)$ Chinese Restaurant Process after $n$ customers are seated, see \cite{CSP}.  
    Consequently, since our chain is stationary, for each $t$,  
    \[ \frac{1}{n}\phi(\upDownChainBiased_n(\lfloor n^2t\rfloor )) = \xi(\iota(\upDownChainBiased_n(\lfloor n^2t\rfloor )))=_d \xi(\iota(\upDownChainBiased_n(0))) \Rightarrow W,\]
    where $W$ has a Beta$(1-\alpha,\alpha)$ distribution, see \cite{CSP}.
   
   Therefore, from Theorem \ref{full statement of convergence result} with $F$ as defined there and $F(0)=_d W$, passing to a subsequence if necessary, and using the Skorokhod representation theorem, we may assume that
    \[ \left( \left(\iota(\upDownChainBiased_n(\lfloor n^2t\rfloor )), \xi(\iota(\upDownChainBiased_n(\lfloor n^2s\rfloor )))\right)\right)_{t,s\geq 0} \overset{a.s.}{\longrightarrow} \left( (\upDownChainBiased(t),F(s))\right)_{t,s\geq 0} \]
    in $D([0,\infty),\setOfOpenSets) \times D([0,\infty), [0, 1])$.
     Fix $t\geq 0$.  Since Feller processes have no fixed discontinuities, $F$ is almost surely continuous at $t$ and, therefore, since convergence in $ D([0,\infty), \setOfOpenSets)$ implies convergence at continuity points, 
    \[\xi(\iota(\upDownChainBiased_n(\lfloor n^2t\rfloor ))) \overset{a.s.}{\longrightarrow} F(t).\]
    Since $ F(t)=_d W$, $\prob(F(t)>0)=1$ and, since 
    \[ \iota(\upDownChainBiased_n(\lfloor n^2t\rfloor ))  \overset{a.s.}{\longrightarrow}  \upDownChainBiased(t),\]
    it follows that $F(t) = _{a.s.}  \xi(\upDownChainBiased(t))$. Consequently, $F(t)$ is a modification of $\xi(\upDownChainBiased(t))$ and since $F$ has a Feller semigroup, so does $\xi(\upDownChainBiased)$.
\end{proof}

\bibliographystyle{plain}      
	\bibliography{LMocrpReferences}

\end{document}